\documentclass[12pt,english]{amsart}
\usepackage[T1]{fontenc}
\usepackage[latin9]{inputenc}
\usepackage{geometry}
\usepackage{babel}
\usepackage{amstext}
\usepackage{amsthm}
\usepackage{amssymb}
\usepackage{esint}
\usepackage{color}
\usepackage[unicode=true,pdfusetitle,
 bookmarks=true,bookmarksnumbered=false,bookmarksopen=false,
 breaklinks=false,pdfborder={0 0 1},backref=false,colorlinks=false]
 {hyperref}
\usepackage{breakurl}

\makeatletter
\numberwithin{equation}{section}
\numberwithin{figure}{section}
\theoremstyle{plain}
\newtheorem{thm}{\protect\theoremname}
  \theoremstyle{definition}
  \newtheorem{defn}[thm]{\protect\definitionname}
  \theoremstyle{remark}
  \newtheorem{rem}[thm]{\protect\remarkname}
  \theoremstyle{plain}
  \newtheorem{prop}[thm]{\protect\propositionname}
  \theoremstyle{plain}
  \newtheorem{lem}[thm]{\protect\lemmaname}
  \theoremstyle{definition}
  \newtheorem{example}[thm]{\protect\examplename}

\usepackage{babel}

  \providecommand{\definitionname}{Definition}
  \providecommand{\examplename}{Example}
  \providecommand{\lemmaname}{Lemma}
  \providecommand{\propositionname}{Proposition}
  \providecommand{\remarkname}{Remark}
\providecommand{\theoremname}{Theorem}

\makeatother

  \providecommand{\definitionname}{Definition}
  \providecommand{\examplename}{Example}
  \providecommand{\lemmaname}{Lemma}
  \providecommand{\propositionname}{Proposition}
  \providecommand{\remarkname}{Remark}
\providecommand{\theoremname}{Theorem}

\begin{document}

\title{Spatial and Temporal white noises \\ under sublinear G-expectation }

\author{JI Xiaojun  and PENG Shige}

%\version{Nov. 06, 2018}

\address{School of Mathematics Shandong University, 250100, Jinan, this research is supported by NSF of China No. L1624032  }
\email{Ji Xiaojun:  jixiaojun@mail.sdu.edu.cn,\,\,\,\, Peng Shige:  peng@sdu.edu.cn }
\begin{abstract}
In the framework of sublinear expectation,
we have introduced  a new type of $G$-Gaussian random fields, which 
contain a type of spatial white noise as a special case. 
Based on this result, we also have introduced a spatial-temporal 
$G$-white noise. 
Different from  the case of linear expectation, in which 
the probability measure need to be known,  under the uncertainty 
of the probability measure, the spatial white noises are intrinsically  
different from the temporal one.  

%we develop a general Gaussian random field under a sublinear expectation
%framework. We first construct a $G$-Gaussian random field whose finite
%dimensional distributions are $G$-normally distributed. A remarkable
%point is that, different from the classical probability theory, a
%$G$-Brownian motion is not a $G$-Gaussian random field. A $G$-Brownian motion is suitable 
%to be as a temporal random field, whereas a random
%Gaussian field is rather suitable for describe spatial uncertainty.
%This new framework of $G$-Gaussian random field can be used to quantitatively
%and efficiently study such type of spatial uncertainties and risks.
%
%By the method of standard partition, we define a special family of
%generating functions $\{G_{\underline{\gamma}}\}_{\underline{\gamma}\in\mathcal{J}_{\Gamma}}$
%as (\ref{g-q}) and then establish the corresponding white noise of
%space type. Stochastic integral of functions in $L^{2}(\mathbb{R}^{d})$
%with respect to the spatial white noise has been established and the
%family of stochastic integrals is still a $G$-Gaussian random field.
%
%Furthermore, we focus on the spatial-temporal G-white noise on the
%sublinear expectation space. Applying the similar method of $G$-It\^o's
%integral, we develop the stochastic calculus with respect to the spatial-temporal
%white noise on $\mathbf{M}_{G}^{2}([0,T]\times\mathbb{R}^{d})$.
%
%$G$-Gaussian random field and $G$-white noise theory can be widely
%used to study the uncertain models in many physical problems, such
%as Ising models.
\medskip
\medskip
\medskip


\end{abstract}

\maketitle

\global\long\def\bbR{\mathbb{R}}

\global\long\def\ccR{\mathcal{R}}

\global\long\def\dlim{\operatorname{\underrightarrow{{\rm lim}}}}

\global\long\def\Ker{\operatorname{\rm Ker}}

\global\long\def\End{\operatorname{\rm End}}

\global\long\def\myint#1#2#3{\int_{#1}^{#2}\sin#3dx}

\section{Introduction}

The axiomatic formulation of probability space $(\Omega,\mathcal{F},P)$
is a powerful and elegant mathematical framework for quantitative
studies of uncertainties. A typical example is to introduce a Wiener
probability measure $P$ on the space of $d$-dimensional continuous
paths $\Omega=C([0,\infty);\mathbb{R}^{d})$, with $\mathcal{F}=\mathcal{B}(\Omega)$.
The canonical process $B_{t}(\omega)=\omega_{t}$, $t\geq0$, is then
a standard Brownian motion, namely, under the Wiener probability measure
$P$, the process $B$ becomes incrementally stable, independent and
continuous. It turns out that any random vector of the form $(B_{t_{1}},\cdots,B_{t_{n}})$
is normally distributed and thus $B$ is a Gaussian process, and all
increments $B_{t+s}-B_{t}$ are independent of $(B_{t_{1}},\cdots,B_{t_{n}}),$
for $t_{1},\cdots,t_{n}\leq t,$ and identically distributed with
respect to $B_{s}$.

But in most cases of our real world, the probability measure
itself is essentially unknown. This type of higher level uncertainty can be described
by a family of probability measures $\{P_{\theta}\}_{\theta\in\Theta}$
such that we are unable to  know ``true one''. In this case
the notion of sublinear expectation 
\[
\mathbb{E}[X]=\sup_{\theta\in\Theta}\int_{\Omega}XdP_{\theta}.
\]
is often used to obtain a robust expectation. This type of formulation induces
a new type of sublinear expectation $\mathbb{E}$ which has been
widely investigated, see, e.g.,  Huber \cite{Huber}(1981), Walley \cite{Walley}(1991).
See also the corresponding coherent risk measures representation of
Arzner et al \cite{ADEH}(1999), Delbaen \cite{Delbaen2002}(2002)
and F\"ollmer \& Schied \cite{Fo-Sch}(2002).

A very intereting question is that can we follow the idea of Kolmogorov-Wiener,
to use the above mentioned idea to construct a basic nonlinear expectation
$\mathbb{E}[\cdot]$, under which the same canonical process $B$
becomes a ``nonlinear'' Brownian motion? Peng \cite{P2007}(2007) introduced
a notion of $G$-Brownian motion which is still the same canonical
process. But the Wiener probability measure $P$ was replaced by the
corresponding sublinear expectation, called G-expectation, such that
$B$ becomes a continuous stochastic process whose increments are
stable and independent. It turns out to be a typical model of stochastic
processes for which the uncertainty of probabilities is dynamically
and intrinsically taken into account. We can also use this $G$-Brownian
motion to generate a very rich class of stochastic processes such
as $G$-OU process, geometric $G$-Brownian motions, through the corresponding
stochastic differential equations driven by this $G$-Brownian motion,
thanks to a generalized It\^o's calculus. 

But it is worth to point out that, for each $t>0$, $B_{t}$ is not
symmetric with respect to $B_{2t}-B_{t}$ in the level of distributions.
We know that $B_{2t}-B_{t}$ is independent from $B_{t}$, but in
general $B_{t}$ is not independent from $B_{2t}-B_{t}$. In fact,
we can check that, although $B_{2t}-B_{t}$ and $B_{t}$ are both
$G$-normally distributed, but the $2$-dimensional random vector
$(B_{2t}-B_{t},B_{t})$ is not $G$-normal. In general $Y$ is independent
from $X$ does not imply that $X$ is independent from $Y$. In fact
the notion of independence becomes a typical temporal relation. For
example, a random sequence $\{X_{i}\}_{i=1}^{\infty}$ is a i.i.d.
sequence under a nonlinear expectation is mainly for the case where
$i\in\mathbb{N}$ is a time index, namely $\{X_{i}\}_{i=1}^{\infty}$
is a time series. How to introduce a space-indexed random fields becomes
an interesting and challenging problem.

In this paper, we introduce a new type of spatial-temporal white noise
in the presence of probability model uncertainties which is a natural
generalization of the classical spatial white noise (see, among others,
Walsh \cite{Walsh}, Dalang \cite{Dalang}, Da Prato and J. Zabczyk
\cite{DaP-Z}). But we emphasis that, in the framework of nonlinear expectation,  
 the space-indexed increments  does not satisfy 
the property of independence. Once the
nonlinear $G$-expectation reduced to its special situation of linear
expectation, the property of independence for the space indexed part
turns out to be true, thanks for the Fubini Lemma for the linear expectation.
The Gaussian property for the time indexed part also turns out to
be true, thanks to the important property of (linear) normal distributions
that independent Gaussian random variables constitute a Gaussian random
vector. 

The objective of this paper is to construct a white noise of space-indexed,
as well as space-time-indexed. We will see that such type of space-indexed
white noise is a special type of Gaussian random field under probability
uncertainty, satisfying distributional invariant of rotations and
translations (see Proposition~\ref{x-invariant}).

An important and interesting point observed in Peng \cite{Peng2011} is
that a $G$-Brownian motion $(B_{t})_{t\geq0}$ is very efficient
to quantify time-indexed random fields, but it is no longer suitable
to describe space-indexed uncertainties. A new notion of Gaussian
process was then introduced in Peng \cite{Peng2011}. Our main point of
view is that under a sublinear expectation, a framework of $G$-Gaussian random
fields is suitable for describing spatial random fields.

In this paper we will systematically develop a new type of spatial
Gaussian random fields under a sublinear expectation space. It is
remarkable that is a finite distribution of a $(B_{t_{1}},\cdots,B_{t_{n}})$
is no longer $G$-normally distributed. 

In many literature, the index $t$ in Brownian motion $B_{t}$ is also
treated as a spatial parameter, i.e., $t=x\in\mathbb{R}$. This method
is widely used in the study of space-time indexed random field. In
fact, this can be justified by Fubini theorem, in which time and space
is exchangeable. However, Fubini theorem does not hold for the nonlinear
expectation framework. So we need to establish a new theoretical framework
in which space and time are treated separately for obtaining a spatial-temporal
random fields, and then develop a new space-time indexed white noise
theory.

We hope that this new framework of spatial-temporal random fields
is useful  in quantitative financial risk
measures, 
%(see F\"ollmer  \cite{Foellmer2014}(2014)), 
complex random networks, stochastic partial
differential equations, non-equlibrium statistical physics, stochastic
control system and games, large scale robust economics, and spatial
and temporal indexed data. 

With the quickly increasing complexity of the internet network, its
uncertainty in probability measures as a space indexed random field
is increasing. There is an urgent need for an exact theoretical framework
of nonlinear expectation in this field.

We establish the general random field under the nonlinear expectation
framework. In particular, we construct spatial random fields and spatial-temporal
white noises, and then the related stochastic calculus theory. Specially,
if the nonlinear expectation degrades to a linear expectation, the
corresponding random field theory and space-time indexed white noise
coincide with classical results. 

This paper is organized as follows. In Section 2, we review some basic
notions and results of nonlinear expectation theory, the notion of
distributions and independence of random vectors, the notion of $G$-normal
distribution, $G$-Brownian motion and their basic properties. In
Section 3, we first provide a generalized Kolmogorov's existence theorem
based from a consistently defined family of finite dimensional distributions
in the framework of nonlinnear expectations. The existence and consistence
of a type of $G$-Gaussian random fields is also provided. In Section
4, we construct the spatial white noise and the corresponding stochastic
integral. The existence and stochastic integral of space-time-indexed
white noise are discussed in Section 5. 

We are aware of  deep research results of F\"ollmer \cite{Foellmer2014} and F\"ollmer \& Kl\"uppelberg
\cite{Foellmer} (2014)  on spatial risk measures and will explore
the relations between the two seemingly very different approaches.

\section{Preliminaries}

In this section, we present some basic notions and properties in nonlinear
expectation. More details can be found in \cite{DHP}, \cite{P2007}, and \cite{P2010}
.

\subsection{Basic notions of nonlinear expectations}

Let $\Omega$ be a given nonempty set and $\mathcal{H}$ be a linear
space of real-valued functions on $\Omega$ such that if $X_{1}$,$\dots$,$X_{d}\in\mathcal{H}$,
then $\varphi(X_{1},X_{2},\dots,X_{d})\in\mathcal{H}$ for each $\varphi\in C_{l.Lip}(\mathbb{R}^{d})$,
where $C_{l.Lip}(\mathbb{R}^{d})$ denotes the linear space of functions
satisfying for each $\ x,y\in\mathbb{R}^{d}$, 
\[
|\varphi(x)-\varphi(y)|\leq C_{\varphi}(1+|x|^{m}+|y|^{m})|x-y|,for\ some\ C_{\varphi}>0,m\in\mathbb{N}\ depending\ on\ \varphi.
\]
$\mathcal{H}$ is considered as the space of random variables. $X=(X_{1},\dots,X_{d})$,
$X_{i}\in\mathcal{H}$, $1\leq i\leq d$, is called a $d$-dimensional
random vector, denoted by $X\in\mathcal{H}^{d}$. 
\begin{defn}
\label{sublinear expectation} A \textbf{sublinear expectation} $\hat{\mathbb{E}}$
on $\mathcal{H}$ is a functional $\mathbb{\hat{E}}:\mathcal{H}\rightarrow\mathbb{R}$
satisfying the following properties: for each $X,Y\in\mathcal{H}$, 
\begin{itemize}
\item[(i)] \textbf{Monotonicity:}\quad{}$\mathbb{\hat{E}}[X]\geq\mathbb{\hat{E}}[Y]\ \ \text{if}\ X\geq Y$; 
\item [{(ii)}] \textbf{Constant preserving:}\quad{}$\mathbb{\hat{E}}[c]=c\ \ \ \text{for}\ c\in\mathbb{R}$; 
\item [{(iii)}] \textbf{Sub-additivity:}\quad{}$\mathbb{\hat{E}}[X+Y]\leq\mathbb{\hat{E}}[X]+\mathbb{\hat{E}}[Y]$; 
\item [{(iv)}] \textbf{Positive homogeneity:}\quad{}$\mathbb{\hat{E}}[\lambda X]=\lambda\mathbb{\hat{E}}[X]\ \ \ \text{for}\ \lambda\geq0$. 

\noindent The triplet $(\Omega,\mathcal{H},\mathbb{\hat{E}})$ is called a \textbf{sublinear
expectation space}. If only (i) and (ii) are satisfied, then $\mathbb{\hat{E}}$
is called a \textbf{nonlinear expectation} and the triplet $(\Omega,\mathcal{H},\mathbb{\hat{E}})$
is called a \textbf{nonlinear expectation space}. 
\item [{(v)}] \textbf{Regularity:} \,\,\, If $\{X_{i}\}_{i=1}^{\infty}\subset\mathcal{H}$
be such that $X_{i}(\omega)\downarrow0$ as $i\rightarrow\infty$, for each $\omega\in\Omega$,
then 
\[
\lim_{i\to\infty}\hat{\mathbb{E}}[X_{i}]=0.
\]
\end{itemize}
\end{defn}
 In this paper we
are mainly interested in sublinear expectations satisfying the regular condition. 
\begin{thm}
Let $\hat{\mathbb{E}}$ be a sublinear expectation on $(\Omega,\mathcal{H})$
satisfying the regularity condition (v) in Definition \ref{sublinear expectation}.
Then there exists a family of unique probability measures $\{P_{\theta}\}_{\theta\in\Theta}$,
such that 
\[
\hat{\mathbb{E}}[X]=\max_{\theta\in\Theta}\int_{\Omega}X(\omega)dP_{\theta},\ \forall X\in\mathcal{H}.
\]

\end{thm}
Let $(\Omega,\mathcal{H},\mathbb{\hat{E}})$ be a nonlinear (resp.
sublinear) expectation space. For each given $d$-dimensional random
vector $X$, we define a functional on $C_{l.Lip}(\mathbb{R}^{d})$
by 
\[
\mathbb{{F}}_{X}[\varphi]:=\mathbb{\hat{E}}[\varphi(X)],\text{ for each }\varphi\in C_{l.Lip}(\mathbb{R}^{d}).
\]
It is easy to verify that $(\mathbb{R}^{d},C_{l.Lip}(\mathbb{R}^{d}),\mathbb{{F}}_{X})$
forms a nonlinear (resp. sublinear) expectation space. $\mathbb{{F}}_{X}$
is called the distribution of $X$. Two $d$-dimensional random vectors
$X_{1}$ and $X_{2}$ defined respectively on nonlinear expectation
spaces $(\Omega_{1},\mathcal{H}_{1},\mathbb{\hat{E}}_{1})$ and $(\Omega_{2},\mathcal{H}_{2},\mathbb{\hat{E}}_{2})$
are called \textbf{identically distributed}, denoted by $X_{1}\overset{d}{=}X_{2}$,
if $\mathbb{{F}}_{X_{1}}=\mathbb{{F}}_{X_{2}}$, i.e., 
\[
\mathbb{\hat{E}}_{1}[\varphi(X_{1})]=\mathbb{\hat{E}}_{2}[\varphi(X_{2})]\text{ for each }\varphi\in C_{l.Lip}(\mathbb{R}^{d}).
\]

Similar to the classical case, Peng \cite{Peng2008a}(2008) gave the
following definition of convergence in distribution. 
\begin{defn}
Let $X_{n}$, $n\geq1$, be a sequence of $d$-dimensional random
vectors defined respectively on nonlinear (resp. sublinear) expectation
spaces $(\Omega_{n},\mathcal{H}_{n},\mathbb{\hat{E}}_{n})$. $\{X_{n}:n\geq1\}$
is said to converge in distribution if, for each fixed $\varphi\in C_{b.Lip}(\mathbb{R}^{d})$,
$\{\mathbb{{F}}_{X_{n}}[\varphi]:n\geq1\}$ is a Cauchy sequence,
where $C_{b.Lip}(\mathbb{R}^{d})$ denotes the set of all bounded
and Lipschitz functions on $\mathbb{R}^{d}$. Define 
\[
\mathbb{{F}}[\varphi]=\lim_{n\rightarrow\infty}\mathbb{{F}}_{X_{n}}[\varphi],
\]
then the triplet $(\mathbb{R}^{d},C_{b.Lip}(\mathbb{R}^{d}),\mathbb{{F}})$
forms a nonlinear (resp. sublinear) expectation space. 
\end{defn}
The following notion of the independence of random variables under
a nonlinear expectation is very useful (see Peng \cite{P2010}(2010)). 
\begin{defn}
Let $(\Omega,\mathcal{H},\mathbb{\hat{E}})$ be a nonlinear expectation
space. An $n$-dimensional random vector $Y$ is said to be \textbf{independent}
from another $m$-dimensional random vector $X$ under the expectation
$\mathbb{\hat{E}}$ if, for each test function $\varphi\in C_{l.Lip}(\mathbb{R}^{m+n})$,
we have 
\[
\mathbb{\hat{E}}[\varphi(X,Y)]=\mathbb{\hat{E}}[\mathbb{\hat{E}}[\varphi(x,Y)]_{x=X}].
\]
Let $\bar{X}$ and $X$ be two $m$-dimensional random vectors on
$(\Omega,\mathcal{H},\hat{\mathbb{E}})$. $\bar{X}$ is called an
independent copy of $X$ if $\bar{X}\overset{d}{=}X$ and $\bar{X}$
is independent from $X$. \end{defn}
\begin{rem}
It is important to note that ``$Y$ is independent from $X$'' does
not imply that ``$X$ is independent from $Y$'' (see Peng \cite{P2010}). 
\end{rem}
For each $p\geq1$ , let $\mathbb{L}^{p}(\Omega)$ be the completion
of $\mathcal{H}$ under the Banach norm $\|X\|_{\mathbb{L}^{p}}:=(\hat{\mathbb{E}}[|X|^{p}])^{\frac{1}{p}}$.
It is easy to verify that $\mathbb{L}^{p}(\Omega)\subset\mathbb{L}^{p^{\prime}}(\Omega)$
for each $1\leq p^{\prime}\leq p$. Since $|\hat{\mathbb{E}}[X]-\hat{\mathbb{E}}[Y]|\leq\hat{\mathbb{E}}[|X-Y|]$,
 $\hat{\mathbb{E}}$ can be continuously extended to the mapping
from $\mathbb{L}^{1}(\Omega)$ to $\mathbb{R}$ and properties $(i)$-$(iv)$
of definition \ref{sublinear expectation} still hold. Moreover, $(\Omega,\mathbb{L}^{1}(\Omega),\hat{\mathbb{E}})$
also forms a sublinear expectation space, which is called a complete
sublinear expectation space.

We say that $X=Y$ in $\mathbb{L}^{p}(\Omega)$ if $\hat{\mathbb{E}}[|X-Y|^{p}]=0$,
and denote by $X\geq Y$ or $Y\leq X$ if $X-Y=(X-Y)^{+}$.

\subsection{G-normal distributions under a sublinear expectation space}

In the rest part of this paper, we focus ourselves to a sublinear
expectation space $(\Omega,\mathcal{H},\hat{\mathbb{E}})$. The following
simple lemma is quite useful: 
\begin{lem}\label{pmLemma} (see Peng \cite{P2010})
If a ramdom variable $\xi\in\mathcal{H}$ satisfies $\hat{\mathbb{E}}[\xi]=\hat{\mathbb{E}}[-\xi]=0$,
then we have 
\[
\hat{\mathbb{E}}[\xi+\eta]=\hat{\mathbb{E}}[\eta],\,\,\,\forall\eta\in\mathcal{H}.
\]
\end{lem}
\begin{defn}
A $d$-dimensional random vector $X=(X_{1},\cdots,X_{d})$ on a sublinear
expectation space $(\Omega,\mathcal{H},\hat{\mathbb{E}})$ is called
(centralized) \textbf{$G$-normally distributed} if 
\[
aX+b\bar{X}\overset{d}{=}\sqrt{a^{2}+b^{2}}X\ \ \ \text{for }a,b\geq0,\ 
\]
where $\bar{X}$ is an independent copy of $X$. 
\end{defn}
We denote by $\mathbb{S}(d)$ the collection of all $d\times d$ symmetric
matrices. The distribution of a $G$-normally distributed random vector
$X$ defined on $(\Omega,\mathcal{H},\hat{\mathbb{E}})$ is uniquely
characterized by a function $G=G_{X}:\mathbb{S}(d)\mapsto\mathbb{R}$
defined as follows: 
\begin{equation}
G_{X}(Q):=\frac{1}{2}\hat{\mathbb{E}}[\langle QX,X\rangle],\ \ Q\in\mathbb{S}(d).\label{G definition}
\end{equation}
It is easy to check that $G$ is a sublinear and continuous function
monotone in $Q\in\mathbb{S}(d)$ in the following sense: for each
$Q,\bar{Q}\in\mathbb{S}(d),$ 
\begin{equation}
\left\{ \begin{array}{l}
G(Q+\bar{Q})\leq G(Q)+G(\bar{Q}),\\
G(\lambda Q)=\lambda G(Q),\quad\forall\lambda\geq0,\\
G(Q)\geq G(\bar{Q}),\qquad if\ Q\geq\bar{Q}.
\end{array}\right.\label{G properties}
\end{equation}

\begin{prop}
\label{G-normal existence} Let $G:\mathbb{S}(d)\rightarrow\mathbb{R}$
be a given sublinear and continuous function, monotone in $Q\in\mathbb{S}(d)$
in the sense of (\ref{G properties}). Then there exists a $G$-normally
distributed $d$-dimensional random vector $X$ on some sublinear
expectation space $(\Omega,\mathcal{H},\hat{\mathbb{E}})$ satisfying
(\ref{G definition}). Moreover, if $\bar{X}$ is also a normally
distributed random variable such that 
\[
G_{X}(Q)=G_{\bar{X}}(Q),\,\,\,\,\,\text{ for any }\,\,Q\in\mathbb{S}(d),%\frac{1}{2}\mathbb{E}[\langleQX,X\rangle]=\frac{1}{2}\mathbb{E}[\langleQ\bar{X},\bar{X}\rangle],\,\,\,\forallQ\in\mathbb{S}(d),
\]
then we have $X\overset{d}{=}\bar{X}$. 
\end{prop}
The function $G_{X}(\cdot)$ associated to the distribution function
of $G$-normal variable $X$ is called the \textbf{generating function}
of $X$. 
\begin{prop}
\label{G heat equation} Let $G$ be given as in Proposition \ref{G-normal existence}
and $X$ be a $d$-dimensional random vector on a sublinear expectation
space $(\Omega,\mathcal{H},\hat{\mathbb{E}})$ such that $X$ is $G$-normally
distributed. For each $\varphi\in C_{l.Lip}(\mathbb{R}^{d})$, define
\begin{equation}
u(t,x):=\hat{\mathbb{E}}[\varphi(x+\sqrt{t}X)],\ (t,x)\in\lbrack0,\infty)\times\mathbb{R}^{d}.\label{G solution}
\end{equation}
Then $u$ is the unique viscosity solution of the $G$-heat equation
\begin{equation}
\partial_{t}u-G(D^{2}u)=0,\ \ u|_{t=0}=\varphi(x)\label{G equation}
\end{equation}

\end{prop}
The following property of a $G$-normally distributed random vector
$\xi$ is easy to check: 
\begin{prop}
Let $\xi$ be a $d$-dimensional $G$-normally distributed random
vector characterized by its generating function 
\[
G_{\xi}(Q):=\frac{1}{2}\hat{\mathbb{E}}[\langle Q\xi,\xi\rangle],\,\,\,Q\in\mathbb{S}(d).
\]
Then for any matrix $K\in\mathbb{R}^{m\times d}$, $K\xi$ is also
an $m$-dimensional $G$-normally distributed random vector. Its corresponding
generating function is 
\begin{equation}
G_{K\xi}(Q)=\frac{1}{2}\hat{\mathbb{E}}[\langle K^{T}QK\xi,\xi\rangle],\,\,\,Q\in\mathbb{S}(m).\label{Kx1}
\end{equation}
\end{prop}
\begin{proof}
Let $\bar{\xi}$ be an independent copy of $\xi$. It is clear that
$K\bar{\xi}$ is also an independent copy of $K\xi$. Thus, for each
$\varphi\in C_{l.Lip}(\mathbb{R}^{m})$, the function $\varphi_{K}(x):=\varphi(Kx)$,
$x\in\mathbb{R}^{d}$, is also a $C_{l.Lip}(\mathbb{R}^{d})$ function.
It follows from the definition of the $G$-normal distribution that, for any $a,b\geq0$,
\[
\hat{\mathbb{E}}[\varphi(aK\xi+bK\bar{\xi})]=\hat{\mathbb{E}}[\varphi_{K}(a\xi+b\bar{\xi})]=\hat{\mathbb{E}}[\varphi_{K}(\sqrt{a^{2}+b^{2}}\xi)]=\hat{\mathbb{E}}[\varphi(\sqrt{a^{2}+b^{2}}K\xi)].
\]
Hence $K\xi$ is also normally distributed. Relation (\ref{Kx1})
is easy to check. 
\end{proof}
By Daniell-Stone Theorem, we have the following robust
representation theorem of sublinear expectations which is quite useful. 
\begin{thm}
(see Peng \cite{P2010}) Let $(\Omega,\mathcal{H},\hat{\mathbb{E}})$
be a sublinear expectation space satisfying the  regular
property. Then there exists a weakly
compact set $\mathcal{P}$ of probability measures defined on $(\Omega,\sigma(\mathcal{H}))$
such that 
\[
\hat{\mathbb{E}}[X]=\max_{P\in\mathcal{P}}E_{P}[X]\text{ for all }X\in\mathcal{H},
\]
where $\sigma(\mathcal{H})$ is the $\sigma$-algebra generated by
all functions in $\mathcal{H}$. $\mathcal{P}$ is called the the
family of probability measures that represents $\hat{\mathbb{E}}$. 
\end{thm}
Let $\mathcal{P}$ be a weakly compact set that represents $\hat{\mathbb{E}}$.
For this $\mathcal{P}$, we define capacity 
\[
c(A):=\sup_{P\in\mathcal{P}}P(A),\ \ A\in\mathcal{B}(\Omega).
\]

A set $A\subset\mathcal{B}(\Omega)$ is called polar if $c(A)=0$.
A property holds ``quasi-surely'' (q.s.) if it holds outside a polar
set. In the following, we do not distinguish two random variables
$X$ and $Y$ if $X=Y\ q.s.$ 
\begin{defn}\label{Def-mod}
Let $\Gamma$ be a set of indices. A family of $\mathbb{R}^{d}$-valued
random vectors $(X_{\gamma})_{\gamma\in\Gamma}$ is called a $d$-dimensional
stochastic process indexed by $\Gamma$, or defined on $\Gamma$, if for each $\gamma\in\Gamma$, $X_\gamma\in\mathcal{H}^d$.
Let $(X_{\gamma})_{\gamma\in\Gamma}$ and $(Y_{\gamma})_{\gamma\in\Gamma}$
be two processes indexed by $\Gamma$. $Y$ is called a quasi-modification
of $X$ if for all $\gamma\in\Gamma$, $X_{\gamma}=Y_{\gamma}\ q.s.$ 
\end{defn}

\section{$G$-Gaussian random fields}

\subsection{A general setting of random fields defined on a nonlinear expectation
space}
\begin{defn}
Let $(\Omega,\mathcal{H},\mathbb{\hat{E}})$ be a nonlinear expectation
space and $\Gamma$ be a parameter set. An $m$-dimensional \textbf{random
field} on $(\Omega,\mathcal{H},\mathbb{\hat{E}})$ is a family of
random vectors $W=(W_{\gamma})_{\gamma\in\Gamma}$ such that $W_{\gamma}\in\mathcal{H}^{m}$
for each $\gamma\in\Gamma$. 
\end{defn}
Let us denote the family of all sets of finite indices by 
\[
\mathcal{J}_{\Gamma}:=\{\underline{\gamma}=(\gamma_{1},\cdots,\gamma_{n}):\forall n\in\mathbb{N},\gamma_{1},\cdots,\gamma_{n}\in\Gamma,\gamma_{i}\neq\gamma_{j}\ for\ i\neq j,1\leq i,j\leq n\}.
\]
Now we give the notion of finite dimensional distribution of the random
field $W$. 
\begin{defn}
For each $\underline{\gamma}=(\gamma_{1},\cdots,\gamma_{n})\in\mathcal{J}_{\Gamma}$,
let $\mathbb{F}_{\underline{\gamma}}$ be a nonlinear expectation
on $(\mathbb{R}^{n\times m},C_{l.Lip}(\mathbb{R}^{n\times m}))$.
We call $(\mathbb{F}_{\underline{\gamma}})_{\underline{\gamma}\in\mathcal{J}_{\Gamma}}$
a system of finite dimensional distributions on $\mathbb{R}^{m}$
with respect to $\Gamma$.

Let $(W_{\gamma})_{\gamma\in\Gamma}$ be an $m$-dimensional random
field defined on a nonlinear expectation space $(\Omega,\mathcal{H},\mathbb{\hat{E}})$.
For each $\underline{\gamma}=(\gamma_{1},\cdots,\gamma_{n})\in\mathcal{J}_{\Gamma}$
and the corresponding ramdom variable $W_{\underline{\gamma}}=(W_{\gamma_{1}},\cdots,W_{\gamma_{n}})$,
we define a functional on $C_{l.Lip}(\mathbb{R}^{n\times m})$ by
\[
\mathbb{F}_{W_{\underline{\gamma}}}[\varphi]=\hat{\mathbb{E}}[\varphi(W_{\underline{\gamma}})],\,\,\,\varphi\in C_{l.Lip}(\mathbb{R}^{n\times m}).
\]
Then the triple $(\mathbb{R}^{m\times n},C_{l.Lip}(\mathbb{R}^{m\times n}),\mathbb{F}_{W_{\underline{\gamma}}})$
constitute a nonlinear expectation space. We call $(\mathbb{F}_{W_{\underline{\gamma}}})_{\underline{\gamma}\in\mathcal{J}_{\Gamma}}$
the family of finite dimensional distributions of $(W_{\gamma})_{\gamma\in\Gamma}$.

Let $(W_{\gamma}^{(1)})_{\gamma\in\Gamma}$ and $({W}_{\gamma}^{(2)})_{\gamma\in\Gamma}$
be two $m$-dimensional random fields defined on nonlinear expectation
spaces $(\Omega_{1},\mathcal{H}_{1},\mathbb{\hat{E}}_{1})$ and $(\Omega_{2},\mathcal{H}_{2},\mathbb{\hat{E}}_{2})$,
respectively. They are said to be identically distributed, denoted
by $(W_{\gamma}^{(1)})_{\gamma\in\Gamma}\overset{d}{=}({W}_{\gamma}^{(2)})_{\gamma\in\Gamma}$,
or simply $W^{(1)}\overset{d}{=}W^{(2)}$, if for each $\underline{\gamma}=(\gamma_{1},\cdots,\gamma_{n})\in\mathcal{J}_{\Gamma}$,
$W_{\underline{\gamma}}^{(1)}\overset{d}{=}{W}_{\underline{\gamma}}^{(2)}$,
i.e., 
\[
\mathbb{\hat{E}}_{1}[\varphi(W_{\underline{\gamma}}^{(1)})]=\mathbb{\hat{E}}_{2}[\varphi({W}_{\underline{\gamma}}^{(2)})]\text{ for each }\varphi\in C_{l.Lip}(\mathbb{R}^{n\times m}).
\]

\end{defn}
For any given $m$-dimensional random field $W=(W_{\gamma})_{\gamma\in\Gamma}$,
the family of its finite dimensional distributions satisfies the following
properties of consistency: 
\begin{description}
\item [{{(1)} Compatibility}] For each $(\gamma_{1},\cdots,\gamma_{n},\gamma_{n+1})\in\mathcal{J}_{\Gamma}$
and $\varphi\in C_{l.Lip}(\mathbb{R}^{n\times m}),$ 
\begin{equation}
\mathbb{F}_{W_{\gamma_{1}},\cdots,W_{\gamma_{n}}}[\varphi]=\mathbb{F}_{W_{\gamma_{1}},\cdots,W_{\gamma_{n}},W_{\gamma_{n+1}}}[\widetilde{\varphi}],\label{F consistent 1}
\end{equation}
where the function $\widetilde{\varphi}$ is a function on $\mathbb{R}^{m\times(n+1)}$
defined by $\widetilde{\varphi}(y_{1},\cdots,y_{n},y_{n+1})=\varphi(y_{1},\cdots,y_{n})$,
for any $y_{1},\cdots,y_{n},y_{n+1}\in\mathbb{R}^{m}$; 
\item [{{(2)} Symmetry}] For each $(\gamma_{1},\cdots,\gamma_{n})\in\mathcal{J}_{\Gamma}$,
$\varphi\in C_{l.Lip}(\mathbb{R}^{n\times m})$ and each permutation
$\pi$ of $\{1,\cdots,n\}$, 
\begin{equation}
\mathbb{F}_{W_{\gamma_{\pi(1)}},\cdots,W_{\gamma_{\pi(n)}}}[\varphi]=\mathbb{F}_{W_{\gamma_{1}},\cdots,W_{\gamma_{n}}}[\varphi_{\sigma}],\label{F consistent 2}
\end{equation}
where we denote $\varphi_{\pi}(y_{1},\cdots,y_{n})=\varphi(y_{\pi(1)},\cdots,y_{\pi(n)})$. 
\end{description}
According to the family of finite dimensional distributions, we can
generalize the classical Kolmogorov's existence theorem to the situation
of a sublinear expectation space, which is a variant of Theorem 3.8
in Peng \cite{Peng2011}. 
\begin{thm}
\label{Kolmogorov existence} Let $\{\mathbb{F}_{\underline{\gamma}},\underline{\gamma}\in\mathcal{J}_{\Gamma}\}$
be a family of finite dimensional distributions on $\mathbb{R}^{m}$ satisfying the compatibility
condition (\ref{F consistent 1}) and the symmetry condition (\ref{F consistent 2}).
Then there exists an $m$-dimensional random field $W=(W_{\gamma})_{\gamma\in\Gamma}$
defined on a nonlinear expectation space $(\Omega,\mathcal{H},\hat{\mathbb{E}})$
whose family of finite dimensional distributions coincides with $\{\mathbb{F}_{\underline{\gamma}},\underline{\gamma}\in\mathcal{J}_{\Gamma}\}$.
If we assume moreover that each $\mathbb{F}_{\underline{\gamma}}$
in $\{\mathbb{F}_{\underline{\gamma}},\underline{\gamma}\in\mathcal{J}_{\Gamma}\}$
is sublinear, then the corresponding expectation $\hat{\mathbb{E}}$
on the space of random variables $(\Omega,\mathcal{H})$ is also sublinear. \end{thm}
\begin{proof}
Denote by $\Omega=(\mathbb{R}^{m})^{\Gamma}$ the space of all functions
$\omega=(\omega_{\gamma})_{\gamma\in\Gamma}$ from $\Gamma$ to $\mathbb{R}^{m}$.
For each $\omega\in\Omega$, we denote by $W=(W_{\gamma}(\omega)=\omega_{\gamma})_{\gamma\in\Gamma}$
and call $W$ the canonical process defined on $\Omega$. The space
of random variables $\mathcal{H}$ is defined by 
\[
\mathcal{H}=L_{ip}(\Omega)=\{\varphi(W_{\gamma_{1}},\cdots,W_{\gamma_{n}}),\forall n\in\mathbb{N},\gamma_{1},\cdots,\gamma_{n}\in\Gamma,\varphi\in C_{l.Lip}(\mathbb{R}^{n\times m})\}.
\]
Then for each random variable $\xi\in\mathcal{H}$ of the form
$\xi=\varphi(W_{\gamma_{1}},\cdots,W_{\gamma_{n}})$, we define the
corresponding nonlinear expectation by 
\[
\hat{\mathbb{E}}[\xi]=\mathbb{F}_{\gamma_{1},\cdots,\gamma_{n}}[\varphi].
\]
Since $\mathbb{F}_{\underline{\gamma}}$, $\underline{\gamma}\in\mathcal{J}_{\Gamma}$
satisfies the consistency conditions (\ref{F consistent 1}) and (\ref{F consistent 2}),
thus the functional $\hat{\mathbb{E}}:\mathcal{H}\mapsto\mathbb{R}$
is consistently defined.

Since, for each $\underline{\gamma}\in{\mathcal{J}}_{\Gamma}$, the
distribution $\mathbb{F}_{\underline{\gamma}}$ is monotone and constant
preserving, hence the expectation $\hat{\mathbb{E}}[\cdot]$ satisfies
the same properties. Namely, $\hat{\mathbb{E}}[\cdot]$ is a nonlinear
expectation defined on $(\Omega,\mathcal{H})$. Obviously, the family
of finite dimensional distributions of $(W_{\gamma})_{\gamma\in\Gamma}$
is $(\mathbb{F}_{\underline{\gamma}})_{\underline{\gamma}\in\mathcal{J}_{\Gamma}}$. Moreover,
if for each ${\underline{\gamma}=(\gamma_{1},\cdots,\gamma_{n})\in\mathcal{J}_{\Gamma}}$,
$\mathbb{F}_{\underline{\gamma}}$ is a sublinear expectation defined
on $(\mathbb{R}^{m\times n},C_{l.Lip}(\mathbb{R}^{m\times n}))$,
then $\hat{\mathbb{E}}$ is a sublinear expectation. The proof of
the sub-additivity of $\hat{\mathbb{E}}[\cdot]$ is as follows: let
$\xi=\varphi(W_{\gamma_{1}},\cdots,W_{\gamma_{i}})$ and $\eta=\psi(W_{\iota_{1}},\cdots,W_{\iota_{j}})$
be two random variables in $\mathcal{H}$, with $\varphi\in C_{l.Lip}(\mathbb{R}^{m\times i})$
and $\psi\in C_{l.Lip}(\mathbb{R}^{m\times j})$. Without loss of
generality, assume that $\gamma_{h}\neq\iota_{k}$, for $h=1,\cdots,i$,
$k=1,\cdots,j$. Then, from the definition of $\hat{\mathbb{E}}$,
\begin{align*}
\hat{\mathbb{E}}[\xi+\eta] & =\mathbb{F}_{\gamma_{1},\cdots,\gamma_{i},\iota_{1},\cdots,\iota_{j}}[\widetilde{\varphi}+\widetilde{\psi}]\leq\mathbb{F}_{\gamma_{1},\cdots,\gamma_{i},\iota_{1},\cdots,\iota_{j}}[\widetilde{\varphi}]+\mathbb{F}_{\gamma_{1},\cdots,\gamma_{i},\iota_{1},\cdots,\iota_{j}}[\widetilde{\psi}]\\
 & =\mathbb{F}_{\gamma_{1},\cdots,\gamma_{i}}[{\varphi}]+\mathbb{F}_{\iota_{1},\cdots,\iota_{j}}[{\psi}]=\hat{\mathbb{E}}[\xi]+\hat{\mathbb{E}}[\eta],
\end{align*}
where we set 
\begin{align*}
&\widetilde{\varphi}(\gamma_{1},\cdots,\gamma_{i},\iota_{1},\cdots,\iota_{j})=\varphi(\gamma_{1},\cdots,\gamma_{i}),\\
&\widetilde{\psi}(\gamma_{1},\cdots,\gamma_{i},\iota_{1},\cdots,\iota_{j})=\psi(\iota_{1},\cdots,\iota_{j}).
\end{align*}
The positive homogeneity of $\hat{\mathbb{E}}$ is directly obtained
from the property of $\mathbb{F}_{\underline{\gamma}}$, ${\underline{\gamma}}\in\mathcal{J}_{\Gamma}$.
The proof is complete. 
\end{proof}

\subsection{Gaussian random fields under a sublinear expectation space}
\begin{defn}
Let $(W_{\gamma})_{\gamma\in\Gamma}$ be an $m$-dimensional random
field on a sublinear expectation space $(\Omega,\mathcal{H},\hat{\mathbb{E}})$.
$(W_{\gamma})_{\gamma\in\Gamma}$ is called an $m$-dimensional $G$-Gaussian
random field if for each $\underline{\gamma}=(\gamma_{1},\cdots,\gamma_{n})\in\mathcal{J}_{\Gamma}$,
the following $n\times m$-dimensional random vector 
\begin{align*}
W_{\underline{\gamma}}=(W_{\gamma_{1}},\cdots,W_{\gamma_{n}})&=(W_{\gamma_{1}}^{(1)},\cdots W_{\gamma_{1}}^{(m)},\cdots,W_{\gamma_{n}}^{(1)},\cdots,W_{\gamma_{n}}^{(m)}),\,\,\,W_{\gamma_{i}}^{(j)}\in\mathcal{H},\,\,\,\, \\  &i=1,\cdots,n,\,\,\,j=1,\cdots,m,
\end{align*}
is $G$-normally distributed. Namely, 
\[
aW_{\underline{\gamma}}+b\bar{W}_{\underline{\gamma}}\overset{d}{=}\sqrt{a^{2}+b^{2}}W_{\underline{\gamma}},\,\,\,\text{ for any }\,\,a,b\geq0,
\]
where $\bar{W}_{\underline{\gamma}}$ is an independent copy of $W_{\underline{\gamma}}$. 
\end{defn}
For each $\underline{\gamma}=(\gamma_{1},\cdots,\gamma_{n})\in\mathcal{J}_{\Gamma}$,
we define 
\[
G_{W_{\underline{\gamma}}}(Q)=\frac{1}{2}\hat{\mathbb{E}}[\langle QW_{\underline{\gamma}},W_{\underline{\gamma}}\rangle],\,\,\,\,Q\in\mathbb{S}(n\times m),
\]
where $\mathbb{S}(n\times m)$ denotes the collection of all $(n\times m)\times(n\times m)$
symmetric matrices. Then $(G_{W_{\underline{\gamma}}})_{\underline{\gamma}\in\mathcal{J}_{\Gamma}}$
constitutes a family of functions: 
\[
G_{W_{\underline{\gamma}}}:\mathbb{S}(n\times m)\mapsto\mathbb{R},\,\,\,\,\underline{\gamma}=(\gamma_{1},\cdots,\gamma_{n}),\,\,\gamma_{i}\in\Gamma,\,\,\,n\in\mathbb{N},
\]
which satisfies the sublinear and monotone properties in the sense
of (\ref{G properties}).

According to Proposition \ref{G-normal existence}, the corresponding
distribution $\mathbb{F}_{W_{\underline{\gamma}}}$ of this $G$-normally
distributed random vector $W_{\underline{\gamma}}$ is uniquely determined
by the function $G_{W_{\underline{\gamma}}}$. Clearly, $\{G_{W_{\underline{\gamma}}}\}_{\underline{\gamma}\in\mathcal{J}_{\Gamma}}$
is a family of monotone sublinear and continuous functions satisfying
the properties of consistency in the following sense: 
\begin{description}
\item [{{(1)} Compatibility}] For any $(\gamma_{1},\cdots,\gamma_{n},\gamma_{n+1})\in\mathcal{J}_{\Gamma}$
and $Q=(q_{ij})_{i,j=1}^{n\times m}\in\mathbb{S}(n\times m)$ 
\begin{equation}
G_{W{\gamma_{1}},\cdots,W_{\gamma_{n}},W_{\gamma_{n+1}}}(\bar{Q})=G_{W{\gamma_{1}},\cdots,W_{\gamma_{n}}}({Q}),\label{G consistent 1}
\end{equation}
where $\bar{Q}=\left(\begin{array}{cc}
Q & 0\\
0 & 0
\end{array}\right)\in\mathbb{S}((n+1)\times m)$; 
\item [{{(2)} Symmetry}] For any permutation $\pi$ of $\{1,\cdots,n\}$,
\begin{equation}
G_{W_{\gamma_{\pi(1)}},\cdots,W_{\gamma_{{\pi}(n)}}}(Q)=G_{W_{\gamma_{1}},\cdots,W_{\gamma_{n}}}({\pi}^{-1}(Q)),\label{G consistent 2}
\end{equation}
where the mapping ${\pi}^{-1}:{\mathbb{S}}(n\times m)\mapsto\mathbb{S}(n\times m)$
is defined by 
\[
\left({\pi}^{-1}(Q)\right)_{ij}=(q_{\pi^{-1}(i)\pi^{-1}(j)}),\,\,\,i,j=1,\cdots,(n\times m),\,\,\,Q=(q_{ij})_{i,j=1}^{n\times m}\in\mathbb{S}(n\times m).
\]

\end{description}
A very inttersting inverse problem is how to construct a $G$-Gaussian
field in a sublinear expectation space if a above type of family of
sublinear functions $(G_{\underline{\gamma}})_{\underline{\gamma}\in\mathcal{J}_{\Gamma}}$
is given. 
\begin{thm}
\label{existence of G.R.F.} Let $(G_{\underline{\gamma}})_{\underline{\gamma}\in\mathcal{J}_{\Gamma}}$
be a family of real valued functions such that, for each $\underline{\gamma}=(\gamma_{1},\cdots,\gamma_{n})\in\mathcal{J}_{\Gamma}$,
the real function $G_{\underline{\gamma}}$ is defined on $\mathbb{S}(n\times m)\mapsto\mathbb{R}$
and satisfies the same monotone and sublinear property of type (\ref{G properties}).
Moreover, this family $(G_{\underline{\gamma}})_{\underline{\gamma}\in\mathcal{J}_{\Gamma}}$
satisfies the same compatibility condition (\ref{G consistent 1})
and symmetry condition (\ref{G consistent 2}). Then there exists
an $m$-dimensional $G$-Gaussian random field $(W_{\gamma})_{\gamma\in\Gamma}$
on a sublinear expectation space $(\Omega,\mathcal{H},\hat{\mathbb{E}})$
such that for each $\underline{\gamma}=(\gamma_{1},\cdots,\gamma_{n})\in\mathcal{J}_{\Gamma}$,
$W_{\underline{\gamma}}=(W_{\gamma_{1}},\cdots,W_{\gamma_{n}})$ is
$G$-normally distributed, i.e. 
\[
G_{W_{\underline{\gamma}}}(Q)=\frac{1}{2}\hat{\mathbb{E}}[\langle QW_{\underline{\gamma}},W_{\underline{\gamma}}\rangle]=G_{\underline{\gamma}}(Q),\,\,\,\text{ for any }Q\in\mathbb{S}(n\times m).
\]
Furthermore, if there exists another Gaussian random field $({\bar{W}}_{\gamma})_{\gamma\in\Gamma}$,
with the same index set $\Gamma$, defined on a sublinear expectation
space $(\bar{\Omega},\bar{\mathcal{H}},\bar{\mathbb{E}})$ such that
for each $\underline{\gamma}\in\mathcal{J}_{\Gamma}$, $\bar{W}_{\underline{\gamma}}$
is $G$-normally distributed with the same generating function, namely,
\[
\frac{1}{2}\bar{\mathbb{E}}[\langle Q\bar{W}_{\underline{\gamma}},\bar{W}_{\underline{\gamma}}\rangle]=G_{\underline{\gamma}}(Q)\,\,\,\text{ for any }Q\in\mathbb{S}(n\times m).
\]
Then we have $W\overset{d}{=}\bar{W}$. \end{thm}
\begin{proof}
For each $\underline{\gamma}={(\gamma_{1},\cdots,\gamma_{n})}\in{\mathcal{J}}_{\Gamma}$,
we denote by $\mathbb{F}_{W_{\underline{\gamma}}}$ the sublinear
distribution uniquely determined by $G_{\underline{\gamma}}$. Since
the family of monotone sublinear functions $G_{W_{\underline{\gamma}}}$
satisfies the conditions of the consistency (\ref{G consistent 1}),
(\ref{G consistent 2}). It follows that the family of $G$-normal
distributions $\mathbb{F}_{W_{\underline{\gamma}}}[\cdot]$ satisfies
the conditions of consistency (\ref{F consistent 1}), (\ref{F consistent 2}).
Thus we can follow the construction procedure in the proof of Theorem~\ref{Kolmogorov existence}
to construct a random field $W=(W_{\gamma})_{\gamma\in\Gamma}$ on
a complete sublinear expectation space $(\Omega,\mathcal{H},\hat{\mathbb{E}})$,
under which the family of finite dimensional distributions coincides
with $\{\mathbb{F}_{W_{\underline{\gamma}}}\}_{\underline{\gamma}\in\mathcal{J}_{\Gamma}}$.
Obviously, $W$ is a Gaussian random field under $\hat{\mathbb{E}}$.

Assume that there exists another $G$-Gaussian random field $\bar{W}$
with the same family of generating functions. Then $W$ and $\bar{W}$
have the same finite-dimensional distributions, which implies $W\overset{d}{=}\bar{W}$.
The proof is complete. 
\end{proof}
For each $p\geq1$, let $\mathbb{L}_{G}^{p}(W)$ be the completion
of $\mathcal{H}$ under the Banach norm $\|\cdot\|_{\mathbb{L}_{G}^{p}}:=\left(\hat{\mathbb{E}}[|\cdot|^{p}]\right)^{1/p}$.
Then the sublinear expectation $\hat{\mathbb{E}}$ can be extended
continuously to $\mathbb{L}_{G}^{p}(W)$ and $(\Omega,\mathbb{L}_{G}^{p}(W),\hat{\mathbb{E}})$
forms a complete sublinear expectation space. 
\begin{rem}
If $\Gamma=\mathbb{R}^{+}$, $W=(W_{\gamma})_{\gamma\in\Gamma}$ becomes
a $G$-Gaussian process which has been studied in Peng \cite{Peng2011}. 
\end{rem}

\section{Spatial white noise under sublinear expectations}

In this section, we formulate a spatial  white noise and then develop
the corresponding stochastic integral. 
\begin{defn}
\label{Gwhitenoise} Let $(\Omega,\mathcal{H},\hat{\mathbb{E}})$
be a sublinear expectation space and $\Gamma=\mathcal{B}_{0}(\mathbb{R}^{d})=\{A\in\mathcal{B}(\mathbb{R}^{d}),\lambda_{A}<\infty\}$,
where $\lambda_{A}$ denotes the Lebesgue measure of $A\in\mathcal{B}(\mathbb{R}^{d})$.
A $1$-dimensional G-Gaussian random field $\mathbb{W}=\{\mathbb{W}_{A},{A\in\Gamma}\}$
is called a $1$-dimensional $G$-white noise if 
\begin{itemize}
 \item[ (i) ] $\hat{\mathbb{E}}[\mathbb{W}^2_{A}]=\overline{\sigma}^2\lambda_A, -\hat{\mathbb{E}}[-\mathbb{W}^2_{A}]=\underline{\sigma}^2\lambda_A$,
for all $A\in\Gamma$; 
\item[ (ii) ] For each $A_{1},A_{2}\in\Gamma$, $A_{1}\cap A_{2}=\emptyset$, we
have 
\begin{align}
 & \hat{\mathbb{E}}[\mathbb{W}_{A_{1}}\mathbb{W}_{A_{2}}]=\hat{\mathbb{E}}[-\mathbb{W}_{A_{1}}\mathbb{W}_{A_{2}}]=0,\label{Gwn1}\\
 & \hat{\mathbb{E}}[(\mathbb{W}_{A_{1}\cup A_{2}}-\mathbb{W}_{A_{1}}-\mathbb{W}_{A_{2}})^{2}]=0,\label{Gwn2}
\end{align}
\end{itemize}
where $0\leq\underline{\sigma}^2\leq\overline{\sigma}^2$ are any given numbers.
\end{defn}

\begin{rem}
The definition implies that the distribution of a $G$-white noise
has only two parameters $\overline{\sigma}^{2}$ and $\underline{\sigma}^{2}$.
In fact we can also set $\overline{\sigma}^{2}=1$, $\underline{\sigma}^{2}\leq1$
to make it as a 1-parameter field. On the other hand, for any given
pair of numbers $0\leq\underline{\sigma}^{2}\leq\overline{\sigma}^{2}$,
one can construct a $1$-dimensional random field $\mathbb{W}=\{\mathbb{W}_{A},{A\in\Gamma}\}$
with $\Gamma=\mathcal{B}_{0}(\mathbb{R}^{d})$ defined on a sublinear
expectation space $(\Omega,\mathcal{H},\hat{\mathbb{E}})$ satisfying
conditions (i)-(ii). 
\end{rem}
We denote by $\mathbb{L}_{G}^{2}(\mathbb{W})$ the completion of $\mathcal{H}$
under the Banach norm $\|\cdot\|_{\mathbb{L}_{G}^{2}}:=\left(\hat{\mathbb{E}}[|\cdot|^{2}]\right)^{1/2}$.
Then the sublinear expectation $\hat{\mathbb{E}}$ can be extended
continuously to $\mathbb{L}_{G}^{2}(\mathbb{W})$ and $(\Omega,\mathbb{L}_{G}^{2}(\mathbb{W}),\hat{\mathbb{E}})$
forms a complete sublinear expectation space.

\subsection{Existence of $G$-white noise}

In the following, we turn to construct a spatial $G$-white noise
on some sublinear expectation space. According to Theorem \ref{existence of G.R.F.},
it is sufficient to define an appropriate family of sublinear and
monotone functions $(G_{\underline{\gamma}})_{\underline{\gamma}\in\mathcal{J}_{\Gamma}}$ such that the $G$-Gaussian random field generated
by $(G_{\underline{\gamma}})_{\underline{\gamma}\in\mathcal{J}_{\Gamma}}$
satisfies conditions (i)-(ii) in Definition \ref{Gwhitenoise} for
$G$-white noise.

For each $\underline{\gamma}=(A_{1},\cdots,A_{n})$, $A_{j}\in\Gamma=\mathcal{B}_{0}(\mathbb{R}^{d})$,
we define a mapping $G_{\underline{\gamma}}(\cdot):\mathbb{S}(n)\mapsto\mathbb{R}$
as follows: 
\begin{eqnarray}
G_{A_{1},\cdots,A_{n}}(Q)=G(\sum_{i,j=1}^{n}q_{ij}\lambda_{A_{i}\cap A_{j}}),\ Q=(q_{ij})_{i,j=1}^n\in\mathbb{S}(n).\label{g-q}
\end{eqnarray}
where $G(a)=\frac{1}{2}\overline{\sigma}^{2}a^{+}-\frac{1}{2}\underline{\sigma}^{2}a^{-}$
for $a\in\mathbb{R}$. Here $\overline{\sigma}^{2}\geq\underline{\sigma}^{2}$
are two given nonnegative parameters. 
\begin{prop}
\label{G-w.n.} For each $\underline{\gamma}=(A_{1},\cdots,A_{n})$,
$A_{j}\in\Gamma=\mathcal{B}_{0}(\mathbb{R}^{d})$, the function $G_{\underline{\gamma}}(\cdot)$
defined as (\ref{g-q}) is a sublinear and monotone real function
defined on $\mathbb{S}(n)$, namely it satisfies the condition~(\ref{G properties}).
Moreover, the family $\{G_{\underline{\gamma}}\}_{\underline{\gamma}\in\mathcal{J}_{\Gamma}}$
satisfies the conditions of consistency given in (\ref{G consistent 1})
and (\ref{G consistent 2}). \end{prop}
\begin{proof}
The property that $G_{\underline{\gamma}}$ satisfies relation (\ref{G properties})
follows from  the monotone and sublinear properties of the function
$G=G(a)$, $a\in\mathbb{R}$, and the property that for each $A_{1},\cdots,A_{n}\in\mathcal{B}_{0}(\mathbb{R}^{d})$,
\begin{equation}\label{nonnegative}
\sum_{i,j=1}^{n}q_{ij}\lambda_{A_{i}\cap A_{j}}\geq0,\,\,\,\,\textrm{ if }Q\geq0. 
\end{equation}

The compatibility of the type (\ref{G consistent 1}) is checked as
follows: let 
\[
\bar{Q}=\left(\begin{array}{cc}
Q & 0\\
0 & 0
\end{array}\right)\in\mathbb{S}(n+1),
\]
and let $A_{1},\cdots,A_{n},A_{n+1}\in\mathcal{B}_{0}(\mathbb{R}^{d})$.
Noting that $q_{i(n+1)}=q_{(n+1)i}=0$, $i=1,\cdots,n+1$, by (\ref{g-q}), 
 we have
\begin{align*}
G_{A_{1},\cdots,A_{n+1}}(\bar{Q})=G_{A_{1},\cdots,A_{n}}(Q),
\end{align*}
which is (\ref{G consistent 1}). Let $\pi(\cdot)$ be a permutation
of $\{1,2,\cdots,n\}$. By (\ref{g-q}), we can check that 
\[
G_{A_{1},\cdots,A_{n}}(Q)=G_{A_{\pi(1)},\cdots,A_{\pi(1)}}(\pi(Q)),
\]
where $\pi(Q)=(q_{\pi(i)\pi(j)})_{i,j=1}^{n}$.
Thus (\ref{G consistent 2}) holds. The proof is complete. 
\end{proof}
Now we present the existence of white noise under the sublinear expectation. 
\begin{thm}
\label{existence of G.W.N.} For each given sublinear and monotone
function 
\[
G(a)=\frac{1}{2}({\overline{\sigma}}^{2}a^{+}-{\underline{\sigma}}^{2}a^{-}),a\in\mathbb{R},
\]
let the family of generating functions $\{G_{\underline{\gamma}}(\cdot)$:
$\underline{\gamma}=(A_{1},\cdots,A_{n})\in\mathcal{J}_{\Gamma}\}$
be defined as in (\ref{g-q}). Then there exists a $1$-dimensional
$G$-Gaussian random field $(\mathbb{W}_{\gamma})_{\gamma\in\Gamma}$
on a complete sublinear expectation space $(\Omega,\mathbb{L}_{G}^{2}(\mathbb{W}),\hat{\mathbb{E}})$
such that, for each $\underline{\gamma}=(A_{1},\cdots,A_{n})\in\mathcal{J}_{\Gamma}$,
$W_{\underline{\gamma}}=(\mathbb{W}_{A_{1}},\cdots,\mathbb{W}_{A_{n}})$
is $G$-normally distributed, i.e.,
\[
G_{\mathbb{W}_{\underline{\gamma}}}(Q)=\frac{1}{2}\hat{\mathbb{E}}[\langle Q\mathbb{W}_{\underline{\gamma}},\mathbb{W}_{\underline{\gamma}}\rangle]=G(\sum_{i,j=1}^{n}q_{ij}\lambda_{A_{i}\cap A_{j}}),\,\,\,\,\text{ for any }Q=(q_{ij})_{i,j=1}^{n}\in\mathbb{S}(n).
\]
Moreover, $(\mathbb{W}_{\gamma})_{\gamma\in\Gamma}$ is also a spatial
$G$-white noise under $(\Omega,\mathbb{L}_{G}^{2}(\mathbb{W}),\hat{\mathbb{E}})$,
namely, conditions (i) and (ii) of Definition~\ref{Gwhitenoise}
are satisfied. 
If $(\bar{\mathbb{W}}_{\gamma})_{\gamma\in\Gamma}$ is another $G$-white noise with the same sublinear function $G$, then 
its family of generating  function $\{G_{\bar{\mathbb{W}}_{\underline{\gamma}}}\}_{\gamma\in \mathcal{J}_\Gamma}$ coincides with the one defined in (\ref{g-q}).

\end{thm}
\begin{proof}
Since the compatibility and symmetry conditions of $\{G_{\underline{\gamma}}\}_{\underline{\gamma}\in\mathcal{J}_{\Gamma}}$
has already been proved, it then follows from Lemma \ref{G-w.n.}
and Theorem \ref{existence of G.R.F.} that the existence and uniqueness
of the $G$-Gaussian random field $\mathbb{W}$ in a complete sublinear
expectation space $(\Omega,\mathbb{L}_{G}^{2}(\mathbb{W}),\hat{\mathbb{E}})$
with the family of generating functions defined by (\ref{g-q}) hold.
We only need to prove that $\mathbb{W}$ is a $G$-white noise. It
suffices to verify that the $G$-random field $\mathbb{W}$ induced
by the family of generating functions defined in (\ref{g-q}) satisfies
conditions (i)-(ii) of Definition~\ref{Gwhitenoise}.

For any $A\in\Gamma$,
\[
  \hat{\mathbb{E}}[\mathbb{W}_A^2]=2G_A(1)=2G(\lambda_A)=\overline{\sigma}^2\lambda_A.
\]
Similarly, we have $\hat{\mathbb{E}}[-\mathbb{W}_A^2]=-\underline{\sigma}^2\lambda_A$, which verifies (i). 
Now consider relation (\ref{Gwn1}) in (ii). Let $A_{1},A_{2}\in\Gamma$
be such that $A_{1}\cap A_{2}=\emptyset$. On the other hand, we can
choose a $2\times2$ symmetric matrix $Q=(q_{ij})_{i,j=1}^{2}$ to
express 
\begin{eqnarray*}
{\hat{\mathbb{E}}}[\mathbb{W}_{A_{1}}\mathbb{W}_{A_{2}}]={\hat{\mathbb{E}}}[\sum_{i,j=1}^{2}q_{ij}\mathbb{W}_{A_{i}}\mathbb{W}_{A_{j}}],
\end{eqnarray*}
where we set 
\[
Q=(q_{ij})_{i,j=1}^{2}=\left[\begin{array}{cc}
0 & \frac{1}{2}\\
\frac{1}{2} & 0
\end{array}\right].
\]
It then follows from (\ref{g-q}) that 
\begin{align*}
\hat{\mathbb{E}}[\mathbb{W}_{A_{1}}\mathbb{W}_{A_{2}}]=\hat{\mathbb{E}}[\sum_{i,j=1}^{2}q_{ij}\mathbb{W}_{A_{i}}\mathbb{W}_{A_{j}}] & =G(\sum_{i,j=1}^{2}q_{ij}\lambda_{A_{i}\cap A_{j}})=0.
\end{align*}
Similarly, we can get $\hat{\mathbb{E}}[-\mathbb{W}_{A_{1}}\mathbb{W}_{A_{2}}]=0$,
thus (\ref{Gwn1}) holds.

To prove (\ref{Gwn2}), let $A_{1},A_{2}\in\Gamma$ be such that $A_{1}\cap A_{2}=\emptyset$
and denote $A_{3}=A_{1}\cup A_{2}$. Then 
\[
\hat{\mathbb{E}}[(\mathbb{W}_{A_{1}}+\mathbb{W}_{A_{2}}-\mathbb{W}_{A_{3}})^{2}]=\hat{\mathbb{E}}[\sum_{i,j=1}^{3}q_{ij}\mathbb{W}_{A_{i}}\mathbb{W}_{A_{j}}],
\]
where we denote 
\[
Q=(q_{ij})_{i,j=1}^{3}=\left[\begin{array}{ccc}
1 & 1 & -1\\
1 & 1 & -1\\
-1 & -1 & 1
\end{array}\right].
\]
We then apply (\ref{g-q}) to get 
\[
\hat{\mathbb{E}}[\sum_{i,j=1}^{3}q_{ij}\mathbb{W}_{A_{i}}\mathbb{W}_{A_{j}}]=G(\sum_{i,j=1}^{3}q_{ij}\lambda_{A_{i}\cup A_{j}})=G(\lambda_{A_{1}}+\lambda_{A_{2}}+\lambda_{A_{1}\cap A_{2}}-2(\lambda_{A_{2}}+\lambda_{A_{1}})]=0,
\]
which yields (\ref{Gwn2}).

Now let $\{{\bar{\mathbb{W}}}_{\gamma}\}_{\gamma\in \Gamma}$ be a $G$-white noise, associated with the same function $G$,   on a complete  sublinear  expectation space $(\bar{\Omega}, \mathbb{L}^2_{G}(\bar{\mathbb{W}}),\bar{\mathbb{E}})$. Then 
for each $\underline{\gamma}=(A_1,\cdots,A_n)\in\mathcal{J}_\Gamma$, $Q=(q_{ij})_{i,j=1}^n\in\mathbb{S}(n)$,
\[
G_{\bar{\mathbb{W}}_{\underline{\gamma}}}(Q)=\frac{1}{2}\hat{\mathbb{E}}[\langle Q\bar{\mathbb{W}}_{\underline{\gamma}},\bar{\mathbb{W}}_{\underline{\gamma}}\rangle]
=G(\sum_{i,j=1}^{n}q_{ij}\lambda_{A_{i}\cap A_{j}})=G_{\underline{\gamma}}(Q).
\]
Thus  $\{{\bar{\mathbb{W}}}_{\gamma}\}_{\gamma\in \Gamma}\overset{d}{=}\{{{\mathbb{W}}}_{\gamma}\}_{\gamma\in \Gamma}$ and the proof is complete.

\end{proof}
We can use a similar approach to prove that, for each $A,B\in\mathcal{B}_{0}(\mathbb{R}^{d})$,
we have 
\[
\hat{\mathbb{E}}(\mathbb{W}_{A\cup B}+\mathbb{W}_{A\cap B}-\mathbb{W}_{A}-\mathbb{W}_{B})^{2}]=0.
\]

\begin{rem}
From the above lemma, it follows directly that $\mathbb{W}_{A_{1}\cup A_{2}}=\mathbb{W}_{A_{1}}+\mathbb{W}_{A_{2}}-\mathbb{W}_{A_{1}\cap A_{2}}$
for any $A_{1},A_{2}\in\Gamma$. Futhermore, if $\{A_{i}\}_{i=1}^{\infty}$
is a mutually disjoint sequence of $\Gamma$ such that $\sum\limits _{i=1}^{\infty}\lambda_{A_{i}}<\infty$,
then 
\[
\mathbb{W}\left(\bigcup_{i=1}^{\infty}A_{i}\right)=\sum_{i=1}^{\infty}\mathbb{W}_{A_{i}}\in\mathbb{L}_{G}^{2}(\mathbb{W}).
\]

\end{rem}
A very important property of spatial $G$-white noise is that it is
invariant under rotations and translations: 
\begin{prop}
\label{x-invariant} For each $p\in\mathbb{R}^{d}$ and $O\in\mathbb{O}(d):=\{O\in\mathbb{R}^{d\times d},\,O^{T}=O^{-1}\}$,
we set 
\[
T_{p,O}(A)=\{Ox+p:x\in A\}.
\]
Then, for each $A_{1},\cdots,A_{n}\in\mathcal{B}_{0}(\mathbb{R}^{d})$,
we have 
\[
(\mathbb{W}_{A_{1}},\cdots,\mathbb{W}_{A_{n}})\overset{d}{=}(\mathbb{W}_{T_{p,O}(A_{1})},\cdots,\mathbb{W}_{T_{p,O}(A_{n})}).
\]
Namely the finite dimensional distributions of $\mathbb{W}$ are invariant
under rotations and translations. \end{prop}
\begin{proof}
It is a direct consequence of the well-know invariance of the Lebesgue
measure under rotations and translations, i.e., 
\[
\lambda_{A}=\lambda_{T_{p,O}(A)},\,\,\,\forall(p,O)\in\mathbb{R}^{d}\times O(d),\,\,\,A\in\mathcal{B}_{0}(\mathbb{R}^{d}).
\]
\end{proof}
According to Theorem \ref{existence of G.W.N.}, we can calculate
and characterize each finite dimensional distribution of the $G$-white
noise $\mathbb{W}$ by the nonlinear heat equations.

Let $\mathbb{W}$ be a $G$-white noise on a sublinear expectation
space $(\Omega,\mathcal{H},\hat{\mathbb{E}})$ with index $\Gamma=\mathcal{B}_{0}(\mathbb{R}^{d})$.
For each $\underline{\gamma}=(A_{1},\cdots,A_{n})\in\mathcal{J}_{\Gamma}$,
$\mathbb{W}_{\underline{\gamma}}=(\mathbb{W}_{A_{1}},\cdots,\mathbb{W}_{A_{n}})$,
$n\in\mathbb{N}$, we have 
\[
G_{\mathbb{W}_{\underline{\gamma}}}(Q)=\frac{1}{2}\hat{\mathbb{E}}[\langle Q{\mathbb{W}}_{\underline{\gamma}},{\mathbb{W}}_{\underline{\gamma}}\rangle]=G_{\underline{\gamma}}(Q),%=\sum_{l=0}^{{2^{n}-2}}G(\sum_{i,j=1}^{n}q_{ij}\delta_{li}^{(n)}\delta_{lj}^{(n)})\lambda_{D_{l}^{(n)}},
\]
where $Q=(q_{ij})\in\mathbb{S}(n)$ and $G(q)=\frac{\overline{\sigma}^{2}}{2}q^{+}-\frac{\underline{\sigma}^{2}}{2}q^{-},\ q\in\mathbb{R}$.
For any given $\varphi\in C_{l.Lip}(\mathbb{R}^{n})$, set 
\[
u(t,x_{1},\cdots,x_{n})=\hat{\mathbb{E}}[\varphi(x_{1}+\sqrt{t}\mathbb{W}_{A_{1}},\cdots,x_{n}+\sqrt{t}\mathbb{W}_{A_{n}})],\,\,\,\,\,(t,x)\in\lbrack0,\infty)\times\mathbb{R}.
\]
Then by Proposition \ref{G heat equation}, $u$ is the unique viscosity
solution of the Cauchy problem of the nonlinear heat equation: 
\begin{equation}
\partial_{t}u-G_{\underline{\gamma}}(D_{xx}^{2}u)=0,\ \ u|_{t=0}=\varphi(x_{1},\cdots,x_{n}).\label{f. d. heat equation}
\end{equation}
We then can obtain the nonlinear expected value by 
\[
\hat{\mathbb{E}}[\varphi(\mathbb{W}_{A_{1}},\cdots,\mathbb{W}_{A_{n}})]=u|_{t=1,x_{1},\cdots,x_{n}=0}.
\]

\subsection{Stochastic integrals with respect to a spatial  white noise}

Now let us  define the stochastic integral with respect to
$G$-white noise in three steps.

 Let $\{{{\mathbb{W}}}_{\gamma}\}_{\gamma\in \Gamma}$, $\Gamma=\mathcal{B}_{0}(\mathbb{R}^{d}),$ be a $1$-dimensional $G$-white noise with parameters $\underline{\sigma}^2,\overline{\sigma}^2$ on a complete  sublinear  expectation space $({\Omega}, \mathbb{L}^2_{G}({\mathbb{W}}),\hat{\mathbb{E}})$.

 \textbf{Step 1:} For each $A\in\Gamma$,
$\lambda_{A}<\infty$, suppose $f:\mathbb{R}^{d}\rightarrow\mathbb{R}$
is an elementary function with the form $f(x)=\textbf{1}_{A}(x),x\in\mathbb{R}^{d}$.
We can define 
\[
\int_{\mathbb{R}^{d}}f(x)\mathbb{W}(dx)=\int_{\mathbb{R}^{d}}\textbf{1}_{A}(x)\mathbb{W}(dx)=\mathbb{W}_{A}.
\]
\textbf{Step 2:} For any simple function $f(x)=\sum\limits _{i=1}^{n}a_{i}\textbf{1}_{A_{i}}(x),n\in\mathbb{N},a_{1},\cdots,a_{n}\in\mathbb{R},A_{1},\cdots,A_{n}\in\Gamma$,
we define 
\[
\int_{\mathbb{R}^{d}}(\sum\limits _{i=1}^{n}a_{i}\textbf{1}_{A_{i}}(x))\mathbb{W}(dx)=\sum\limits _{i=1}^{n}a_{i}\int_{\mathbb{R}^{d}}\textbf{1}_{A_{i}}(x)\mathbb{W}(dx)=\sum\limits _{i=1}^{n}a_{i}\mathbb{W}_{A_{i}}.
\]
\textbf{Step 3:} Let $L^{2}(\mathbb{R}^{d})$ denote the space of
all functions $f:\mathbb{R}^{d}\rightarrow\mathbb{R}$ such that $\|f\|_{L^{2}}^{2}=\int_{\mathbb{R}^{d}}|f(x)|^{2}dx<\infty.$
By the following lemma, we can extend the stochastic integral to $L^{2}(\mathbb{R}^{d})$. 
\begin{lem}
If $f:\mathbb{R}^{d}\rightarrow\mathbb{R}$ is a simple function,
then 
\[
\hat{\mathbb{E}}[|\int_{\mathbb{R}^{d}}f(x)\mathbb{W}(dx)|^{2}]=\overline{\sigma}^{2}\|f\|_{L^{2}}^{2}.
\]
\end{lem}
\begin{proof} Without loss of generality, we
suppose $f(x)=\sum\limits _{i=1}^{n}a_{i}\textbf{1}_{A_{i}}(x)$, where
$a_{1},\cdots,a_{n}\in\mathbb{R},A_{1},\cdots,A_{n}\in\Gamma$ and
$A_{i}\cap A_{j}=\emptyset$ if $i\neq j$. Then 
\begin{align*}
\hat{\mathbb{E}}[|\int_{\mathbb{R}^{d}}f(x)\mathbb{W}(dx)|^{2}] & =\hat{\mathbb{E}}[|\sum\limits _{i=1}^{n}a_{i}\mathbb{W}_{A_{i}}|^{2}]=\hat{\mathbb{E}}[\sum\limits _{i=1}^{n}a_{i}^{2}\mathbb{W}_{A_{i}}^{2}]\\
 & \leq\sum\limits _{i=1}^{n}a_{i}^{2}\hat{\mathbb{E}}[\mathbb{W}_{A_{i}}^{2}]=\overline{\sigma}^{2}\sum\limits _{i=1}^{n}a_{i}^{2}\lambda_{{A_{i}}}=\overline{\sigma}^{2}\|f\|_{L^{2}}^{2}.
\end{align*}

\end{proof}

This lemma implies that the linear mapping 
\[
\int_{\mathbb{R}^{d}}f(x)\mathbb{W}(dx):L^{2}(\mathbb{R}^{d})\mapsto\mathbb{L}_{G}^{2}(\mathbb{W})
\]
is contract. We then can continuously extend this mapping to the the
whole domain of $L^{2}(\mathbb{R}^{d})$. We still denote this mapping
by$\int_{\mathbb{R}^{d}}f(x)\mathbb{W}(dx)$ and call it integral
of $f\in L^{2}(\mathbb{R}^{d})$ with respect to the spatial  white
noise $\mathbb{W}.$

For each $B\subset\mathbb{R}^{d}$, $f\in L^{2}(\mathbb{R}^{d})$,
we can also define 
\[
\int_{B}f(x)\mathbb{W}(dx)=\int_{\mathbb{R}^{d}}f(x)\textbf{1}_{B}(x)\mathbb{W}(dx).
\]

\begin{thm}
\label{th1} Let $\Gamma=\mathcal{B}_{0}(\mathbb{R}^{d})$ and $\mathbb{W}=\{\mathbb{W}_{A},{A\in\Gamma}\}$
be a 1-dimensional $G$-white noise on a sublinear expectation space
$(\Omega,\mathcal{H},\hat{\mathbb{E}})$. Then $\{\int_{\mathbb{R}^{d}}f(x)\mathbb{W}(dx):f\in L^{2}(\mathbb{R}^{d})\}$
is a $G$-Gaussian random field. 
\end{thm}
To prove this theorem, we need the following lemma to show that the
convergence of a sequence of $G$-normally distributed random vectors
is still $G$-normally distributed. 
\begin{lem}
\label{normal} Let $\{X_{n}\}_{n=1}^{\infty}$ be a sequence of $m$-dimensional
$G$-normally distributed random variables on a sublinear expectation
space $(\Omega,\mathcal{H},\hat{\mathbb{E}})$. If for some $p\geq1$,
$\lim\limits _{n\rightarrow\infty}\hat{\mathbb{E}}[|X_{n}-X|^{p}]=0$,
then $X$ is also $G$-normally distributed. \end{lem}
\begin{proof}
As for some $p\geq1$, $\lim\limits _{n\rightarrow\infty}\hat{\mathbb{E}}[|X_{n}-X|^{p}]=0$,
we can deduce that $\{X_{n}\}_{n=1}^{\infty}$ converges in distribution
by 
\[
|\mathbb{{F}}_{X_{n}}[\varphi]-\mathbb{{F}}_{X}[\varphi]|=|\mathbb{\hat{E}}[\varphi(X_{n})]-\mathbb{\hat{E}}[\varphi({X})]|\leq C_{\varphi}\mathbb{\hat{E}}[|X_{n}-X|]\leq C_{\varphi}\mathbb{\hat{E}}[|X_{n}-X|^{p}]^{1/p}\rightarrow0,
\]
as $n\rightarrow\infty$, where $C_{\varphi}$ is the Lipschitz constant of $\varphi\in C_{b.Lip}(\mathbb{R}^{m})$.
Assume that $\bar{X}_{n}$ and $\bar{X}$ are the independent copies
of $X_{n}$ and $X$, respectively. Then, for each $\varphi\in C_{b.Lip}(\mathbb{R}^{m})$, $a,b\geq0$,
we have 
\[
\mathbb{\hat{E}}[\varphi(aX_{n}+b\bar{X}_{n})]=\mathbb{\hat{E}}[\varphi(\sqrt{a^{2}+b^{2}}X_{n})]\rightarrow\mathbb{\hat{E}}[\varphi(\sqrt{a^{2}+b^{2}}X)],\ as\ n\rightarrow\infty.
\]
On the other hand, setting $\varphi_{n}(y)={\hat{\mathbb{E}}}[\varphi(y+bX_{n})]$
and $\bar{\varphi}(y)={\hat{\mathbb{E}}}[\varphi(y+bX)]$, it is easy
to show that $\varphi_{n}(y)$ is a uniform Lipschitz function with
the same Lipschitz constant as $\varphi$, such that $|\varphi_{n}(y)-\bar{\varphi}(y)|\leq bC_{\varphi}{\hat{\mathbb{E}}}[|X-X_{n}|]$,
for all $y\in\mathbb{R}^{m}$. Thus 
\begin{align*}
 & |\mathbb{\hat{E}}[\varphi(aX_{n}+b\bar{X}_{n})]-\mathbb{\hat{E}}[\varphi(aX+b\bar{X})]|\\
= & |\mathbb{\hat{E}}[\varphi_{n}(aX_{n})]-\mathbb{\hat{E}}[\bar{\varphi}(aX)]|\\
\leq & |\mathbb{\hat{E}}[\varphi_{n}(aX_{n})]-\mathbb{\hat{E}}[\varphi_{n}(aX)]|+\mathbb{\hat{E}}[|\varphi_{n}(aX)-\bar{\varphi}(aX)|]\to0,\,\,\,\,\text{ as }n\to\infty.
\end{align*}
Hence, $\mathbb{\hat{E}}[\varphi(aX+b\bar{X})]|=\mathbb{\hat{E}}[\varphi(\sqrt{a^{2}+b^{2}}X)]$,
for each $\varphi\in C_{b.Lip}(\mathbb{R}^{m})$, which completes
the proof. 
\end{proof}
\begin{proof}
[Proof of Theorem \ref{th1}]For any $n\in\mathbb{N}$, $f_{1},\cdots,f_{n}\in L^{2}(\mathbb{R}^{d})$,
there exists a sequence of simple functions $\{(f_{1}^{m},\cdots,f_{n}^{m}),m\geq1\}$
such that $\lim\limits _{m\rightarrow\infty}f_{i}^{m}=f_{i}$ in $L^{2}(\mathbb{R}^{d})$,
$i=1,\cdots,n$, and 
\[
(\int_{\mathbb{R}^{d}}f_{1}(x)\mathbb{W}(dx),\cdots,\int_{\mathbb{R}^{d}}f_{n}(x)\mathbb{W}(dx))=\lim\limits _{m\rightarrow\infty}(\int_{\mathbb{R}^{d}}f_{1}^{m}(x)\mathbb{W}(dx),\cdots,\int_{\mathbb{R}^{d}}f_{n}^{m}(x)\mathbb{W}(dx)).
\]
By Proposition 1.6 in Chapter II in Peng \cite{P2010}, it is easy to verify
that 
\[(\int_{\mathbb{R}^{d}}f_{1}^{m}(x)\mathbb{W}(dx),\cdots,\int_{\mathbb{R}^{d}}f_{n}^{m}(x)\mathbb{W}(dx))
\]
is $G$-normally distributed. Using Lemma \ref{normal}, we deduce
that 
\[
(\int_{\mathbb{R}^{d}}f_{1}(x)\mathbb{W}(dx),\cdots,\int_{\mathbb{R}^{d}}f_{n}(x)\mathbb{W}(dx))
\]
is $G$-normally distributed which completes the proof. \end{proof}
\begin{example}
Let $\{\mathbb{W}_{A},A\in\mathcal{B}_{0}(\mathbb{R})\}$ be a 1-dimensional
$G$-white noise. Define $\mathbb{B}_{t}=\mathbb{W}([0,t])$, $t\in\mathbb{R}_{+}$,
then 
\[
\hat{\mathbb{E}}[\mathbb{B}_{t}\mathbb{B}_{s}]=\overline{\sigma}^{2}\lambda_{1}([0,t]\cap[0,s])=\overline{\sigma}^{2}(s\wedge t).
\]
Different from the classical case, $\mathbb{B}_{t}$ is no longer
a $G$-Brownian motion although $\mathbb{B}_{t}\overset{d}{=}N(\{0\}\times[\underline{\sigma}^{2}t,\overline{\sigma}^{2}t])$
for each $t\geq0$. 
\end{example}

\subsection{The continuity of spatial $G$-white noise}

For each $x\in\mathbb{R}^{d}$, let $\Gamma$ denote the set of hypercubes
in $\mathbb{R}^{d}$ that contains the origin and $t$ as its two
extremal vertices in the sense that 
\begin{equation}
\Gamma=\{(0\wedge x,0\vee x]:=(0\wedge x_{1},0\vee x_{1}]\times\cdots\times(0\wedge x_{d},0\vee x_{d}]:\forall x=(x_{1},\cdots,x_{d})\in\mathbb{R}^{d}\}.\label{Gamma}
\end{equation}
Suppose $(V_{\gamma})_{\gamma\in\Gamma}$ is a random field in a sublinear
expectation space $(\Omega,\mathcal{H},\hat{\mathbb{E}})$. Define
$V_{x}:=V_{(0\wedge x,0\vee x]}$, $x\in\mathbb{R}^{d}$. Then $(V_{x})_{x\in\mathbb{R}^{d}}$
is a random field indexed by $x\in\mathbb{R}^{d}$. In particular,
we derive from a spatial random field $\mathbb{W}_{A},$ $A\in\mathcal{B}(\mathbb{R}^{d})$.

For the spatial white noise $(\mathbb{W}_{A})_{A\in\mathcal{B}(\mathbb{R}^{d})}$,
the corresponding random field $(\mathbb{W}_{x})_{x\in\mathbb{R}^{d}}$
is called $\mathbb{R}^{d}$-indexed spatial white noise. We will see
in this subsection that, just like in the classical situation, this
$\mathbb{R}^{d}$-indexed spatial white noise $(\mathbb{W}_{x})_{x\in\mathbb{R}^{d}}$
has also a quasi-continuous modification in the sense of Definition~\ref{Def-mod}.
%{\color{red}
%\begin{defn}
%Let $I$ be a set of indices. $(X_{x})_{x\in I}$ and $(Y_{x})_{x\in I}$
%be two random fields indexed by $I$. We say that $Y$ is a quasi-modification
%of $X$ if for each $x\in I$, $Y_{x}=X_{x}$, quasi surely. 
%\end{defn}
%same as Definition 12
%}

We recall the following generalized Kolmogorov's continuity criterion
for a random field indexed by $\mathbb{R}^{d}$ (see Denis et al. \cite{DHP}).
\begin{thm}
\label{Kolmogorov criterion1} Let $p>0$ be given and $(V_{x})_{x\in\mathbb{R}^{d}}$
be a random field such that $V_{x}\in{\color{red}\mathbb{L}^{p}(\Omega)}$ for each $x\in\mathbb{R}^{d}.$ Assume
that there exist constants $c,\varepsilon>0$ such that if there exist
positive constants $c,\varepsilon,p$ satisfying 
\[
\hat{\mathbb{E}}[|V_{x}-V_{y}|^{p}]\leq c|x-y|^{d+\varepsilon},\,\,\,\,\text{ for all }\,\,x,y\in\mathbb{R}^{d},
\]
then $(V_{x})_{x\in[0,T]}$ has a continuous quasi-modification $(\tilde{V}_{x})_{x\in\mathbb{R}^{d}}$
such that for each $\alpha\in[0,\varepsilon/p)$, 
\[
\hat{\mathbb{E}}[(\sup\limits _{x\neq y}\frac{|\tilde{V}_{x}-\tilde{V}_{y}|}{|x-y|^{\alpha}})^{p}]<\infty.
\]
Hence, the paths of $\tilde{V}$ are quasi-surely H\"older continuous
of order $\alpha$, $\alpha\in[0,\varepsilon/p)$, in the sense that
there exists a Borel set $N$ in with capacity 0 such that for all
$\omega\in N^{c}$, the mapping $x\rightarrow\tilde{V}_{x}(\omega)$
is Holder continuous of order $\alpha$. 
\end{thm}
We have the following generalized Kolmogorov's criterion for spatial
white noise $(\mathbb{W}_{x})_{x\in\mathbb{R}^{d}}$
\begin{thm}\label{modification1}
For each $x\in\mathbb{R}^{d}$, let $\mathbb{W}_{x}:=\mathbb{W}_{(0,x]},$
be a $1$-dimensional spatial $G$-white noise on the sublinear expectation
space $(\Omega,\mathcal{H},\hat{\mathbb{E}})$. Then there exists
a continuous quasi-modification $\tilde{\mathbb{W}}$ of $\mathbb{W}$. \end{thm}
\begin{proof}
For notational simplicity, we only prove the case that $d=2$. For
each $x=(x_{1},x_{2}), y=(y_{1},y_{2})\in\mathbb{R}^{2}$, without
loss of generality, assume that $0\leq x_{i}\leq y_{i}$, $i=1,2$.
Then, by (\ref{g-q}), 
\[
\hat{\mathbb{E}}[|\mathbb{W}_{y}-\mathbb{W}_{x}|^{2}]=\overline{\sigma}^{2}\lambda_{(0,y]/(0,x]}=\overline{\sigma}^{2}(y_{1}y_{2}-x_{1}x_{2}).
\]
Thus $\mathbb{W}_{y}-\mathbb{W}_{x}$ is $G$-normally distributed
and $\mathbb{W}_{y}-\mathbb{W}_{x}\overset{d}{=}N(\{0\}\times[\underline{\sigma}^{2}(y_{1}y_{2}-x_{1}x_{2}),\overline{\sigma}^{2}(y_{1}y_{2}-x_{1}x_{2})]$.
It follows that 
\[
\hat{\mathbb{E}}[|\mathbb{W}_{y}-\mathbb{W}_{x}|^{6}]=15\overline{\sigma}^{6}|y_{1}y_{2}-x_{1}x_{2}|^{3}.
\]
Since 
\[
y_{1}y_{2}-x_{1}x_{2}\leq(y_{1}\vee y_{2})(|y_{1}-x_{1}|+|y_{2}-x_{2}|)\leq\sqrt{2}(y_{1}\vee y_{2})|y-x|,
\]
we have 
\[
\hat{\mathbb{E}}[|\mathbb{W}_{y}-\mathbb{W}_{x}|^{6}]\leq30\sqrt{2}\overline{\sigma}^{6}(y_{1}\vee y_{2})|y-x|^{3}.
\]
We then apply Theorem \ref{Kolmogorov criterion1} and obtain that
$\mathbb{W}$ has a continuous modification $\tilde{\mathbb{W}}$. 
\end{proof}

\section{Spatial and temporal white noise under G-sublinear expectation}

\subsection{Basic notions and main properties}

In this section, we study a new type spatial-temporal white noise
which is a random field defined on the following index set 
\[
\Gamma=\{[s,t)\times A:0\leq s\leq t<\infty,\,\,A\in\mathcal{B}_{0}(\mathbb{R}^{d})\}.
\]
Specifically, we have 
\begin{defn}
\label{tx-Gwhitenoise} A random field $\{\mathbf{W}([s,t)\times A)\}_{([s,t)\times A)\in\Gamma}$
on a sublinear expectation space $(\Omega,\mathcal{H},\hat{\mathbb{E}})$
is called a 1-dimensional spatial-temporal $G$-white noise if it
satisfies the following conditions: \end{defn}
\begin{itemize}
\item [{(i)}] For each fixed $[s,t)$, the ramdom field $\{\mathbf{W}([s,t)\times A)\}_{{A\in\mathcal{B}_{0}(\mathbb{R}^{d})}}$
is a $1$-dimensional spatial white-noise that has the same finite-dimensional
distributions as \\ $(\sqrt{t-s}\mathbb{W}_{A}){}_{{A\in\mathcal{B}_{0}(\mathbb{R}^{d})}}$; 
\item [{(ii)}] For any $r\leq s\leq t$, $A\in\mathcal{B}_{0}(\mathbb{R}^{d})$,
$\mathbf{W}([r,s)\times A)+\mathbf{W}([s,t)\times A)=\mathbf{W}([r,t)\times A)$; 
\item [{(iii)}] $\mathbf{W}([s,t)\times A)$ is independent of $\left(\mathbf{W}([s_{1},t_{1})\times A_{1}),\cdots,\mathbf{W}([s_{n},t_{n})\times A_{n})\right)$,
if $t_{i}<s$ and $A_{i}\in\mathcal{B}_{0}(\mathbb{R}^{d})$, for
$i=1,\cdots,n$, 
\end{itemize}
where $(\mathbb{W}_{A})_{{A\in\mathcal{B}_{0}(\mathbb{R}^{d})}}$
is a 1-dimensional $G$-white noise. 
\begin{rem}
It is important to mention that $\{\mathbf{W}([s,t)\times A),0\leq s\leq t<\infty,A\in\mathcal{B}_{0}(\mathbb{R}^{d})\}$
is no longer a $G$-Gaussian random field. 
\end{rem}
We now present the existence of the spatial-temporal G-white noise
and the construction of the corresponding expectation and conditional
expectation.

Suppose $\Omega=(\mathbb{R})^{\Gamma}$ and for each $\omega\in\Omega$,
define the canonical process $(\mathbf{W}_{\gamma})_{\gamma\in\Gamma}$
by 
\[
\mathbf{W}([s,t)\times A)(\omega)=\omega([s,t)\times A),\ \forall0\leq s\leq t<\infty,\ A\in\mathcal{B}_{0}(\mathbb{R}^{d}).
\]
For each $T\geq0$, set $\mathcal{F}_{T}=\sigma\{\mathbf{W}([s,t)\times A),0\leq s\leq t\leq T,A\in\mathcal{B}(\mathbb{R}^{d})\}$,
$\mathcal{F}=\bigvee\limits _{T\geq0}\mathcal{F}_{T}$, and 
\begin{align*}
L_{ip}(\mathcal{F}_{T})= & \{\varphi(\mathbf{W}([s_{1},t_{1})\times A_{1}),\cdots,\mathbf{W}([s_{n},t_{n})\times A_{n})),\,\,\text{ for any }n\in\mathbb{N},\,\,s_{i}\leq t_{i}\leq T,\,
\\ 
& i=1,\cdots,n,\,\,\,
 A_{1},\cdots,A_{n}\in\mathcal{B}_{0}(\mathbb{R}^{d}),\varphi\in C_{l.Lip}(\mathbb{R}^{n})\}.
\end{align*}
We also set $L_{ip}(\mathcal{F})=\cup_{n=1}^{\infty}L_{ip}(\mathcal{F}_{n})$.
For each $X\in L_{ip}(\mathcal{F})$, without loss of generality,
we assume $X$ has the form 
\begin{align*}
X=&\varphi(\mathbf{W}([0,t_{1})\times A_{1}),\cdots,\mathbf{W}([0,t_{1})\times A_{m}),\cdots,\mathbf{W}([t_{n-1},t_{n})\times A_{1}),\cdots,\\
&\ \ \ \mathbf{W}([t_{n-1},t_{n})\times A_{m})),
\end{align*}
where $0=t_{0}\leq t_{1}<\cdots<t_{n}<\infty$, $\{A_{1},\cdots,A_{m}\}\subset\mathcal{B}_{0}(\mathbb{R}^{d})$
are mutually disjoint, and $\varphi\in C_{l.Lip}(\mathbb{R}^{m\times n})$.
Then we can define the sublinear expectation for $X$ in the following
procedure.

Let $\{G_{A_{1},\cdots,A_{m}}:\forall m\in\mathbb{N},0\leq s<t<\infty,A_{1},\cdots,A_{m}\in\mathcal{B}_{0}(\mathbb{R}^{d})\}$
be a family of monotonic and sublinear functions defined as (\ref{g-q}).
By Proposition \ref{G-normal existence}, there exists a sequence
of $m$-dimensional $G$-normally distributed random vectors $\{\xi_{1},\cdots,\xi_{n}\}$,
$\xi_{i}=(\xi_{i}^{(1)},\cdots,\xi_{i}^{(m)})$, $1\leq i\leq n$,
on a sublinear expectation $(\tilde{\Omega},\tilde{\mathcal{H}},\tilde{\mathbb{E}})$
such that $\xi_{i+1}$ is independent of $(\xi_{1},\cdots,\xi_{i})$
for each $i=1,2,\cdots,n-1$, and 
\[
G_{A_{1},\cdots,A_{m}}(Q)=\frac{1}{2}\tilde{\mathbb{E}}[\langle Q\xi_{i},\xi_{i}\rangle]=G(\sum_{i=1}^{n}q_{ii}\lambda_{A_{i}}),\forall Q\in\mathbb{S}(m),
\]
where $G(q)=\frac{\overline{\sigma}^{2}}{2}q^{+}-\frac{\underline{\sigma}^{2}}{2}q^{-}$,
$q\in\mathbb{R}$.

Define 
\begin{align*}
\hat{\mathbb{E}}[X] & =\hat{\mathbb{E}}[\varphi(\mathbf{W}([0,t_{1})\times A_{1}),\cdots,\mathbf{W}([0,t_{1})\times A_{m}),\cdots,\mathbf{W}([t_{n-1},t_{n})\times A_{1}),\cdots,\\
& \ \ \ \ \ \ \ \ \ \  \mathbf{W}([t_{n-1},t_{n})\times A_{m}))]\\
 & =\tilde{\mathbb{E}}[\varphi(\sqrt{t_{1}}\xi_{1}^{(1)},\cdots,\sqrt{t_{1}}\xi_{1}^{(m)},\cdots,\sqrt{(t_{n}-t_{n-1})}\xi_{n}^{(1)},\cdots,\sqrt{(t_{n}-t_{n-1})}\xi_{n}^{(m)})].
\end{align*}
and the related conditional expectation of $X$ under $\mathcal{F}_{t}$, $t_{j}\leq t<t_{j+1}$, denoted by $\hat{\mathbb{E}}[X|\mathcal{F}_{t}] $, is defined by
\begin{align*}
& \hat{\mathbb{E}}[\varphi(\mathbf{W}([0,t_{1})\times A_{1}),\cdots,\mathbf{W}([0,t_{1})\times A_{m}),\cdots,\mathbf{W}([t_{n-1},t_{n})\times A_{1}),\cdots,\\
&\ \ \ \ \ \ \ \mathbf{W}([t_{n-1},t_{n})\times A_{m}))|\mathcal{F}_{t}]\\
 & =\psi(\mathbf{W}([0,t_{1})\times A_{1}),\cdots,\mathbf{W}([0,t_{1})\times A_{m}),\cdots,\mathbf{W}([t_{j-1},t_{j})\times A_{1}),\cdots,\\
 &\ \ \ \ \ \ \ \mathbf{W}([t_{j-1},t_{j})\times A_{m})).
\end{align*}
where 
\begin{align*}
\psi(x_{11},\cdots,x_{jm})  =\tilde{\mathbb{E}}[\varphi(x_{11},&\cdots,x_{jm},\sqrt{t_{j+1}-t_{j}}\xi_{j+1}^{(1)},\cdots,\sqrt{t_{j+1}-t_{j}}\xi_{j+1}^{(m)}, \\ & \cdots,\sqrt{t_{n}-t_{n-1}}\xi_{n}^{(1)},
  \cdots,\sqrt{t_{n}-t_{n-1}}\xi_{n}^{(m)})].
\end{align*}
It is easy to check that $\hat{\mathbb{E}}[\cdot]$ defines a sublinear
expectation on $L_{ip}(\mathcal{F})$ and the canonical process $(\mathbf{W}_{\gamma})_{\gamma\in\Gamma}$
is a $1$-dimensional spatial-temporal white noise on $(\Omega,L_{ip}(\mathcal{F}),\hat{\mathbb{E}}).$

For each $p\geq1$, $T\geq0$, we denote by $\mathbf{L}_{G}^{p}(\mathbf{W}_{[0,T]})(resp.\ \mathbf{L}_{G}^{p}(\mathbf{W}))$
the completion of $L_{ip}(\mathcal{F}_{T})(resp.\ L_{ip}(\mathcal{F}))$
under the form $\|X\|_{p}:=(\hat{\mathbb{E}}[|X|^{p}])^{1/p}$. The
conditional expectation $\hat{\mathbb{E}}[\cdot\left|\mathcal{F}_{t}\right]:L_{ip}(\mathcal{F})\rightarrow L_{ip}(\mathcal{F}_{t})$
is a continuous mapping under $\|\cdot\|_{p}$ and can be extended
continuously to the mapping $\mathbf{L}_{G}^{p}(\mathbf{W})\rightarrow\mathbf{L}_{G}^{p}(\mathbf{W}_{[0,t]})$
by 
\[
|\hat{\mathbb{E}}[X\left|\mathcal{F}_{t}\right]-\hat{\mathbb{E}}[Y\left|\mathcal{F}_{t}\right]|\leq\hat{\mathbb{E}}[|X-Y|\left|\mathcal{F}_{t}\right].
\]
It is easy to verify the conditional expectation $\hat{\mathbb{E}}[\cdot|\mathcal{F}_{t}]$
satisfies the following properties. We omit the proof here. 
\begin{prop}
The conditional expectation $\hat{\mathbb{E}}[\cdot\left|\mathcal{F}_{t}\right]:\mathbf{L}_{G}^{p}(\mathbf{W})\rightarrow\mathbf{L}_{G}^{p}(\mathbf{W}_{[0,t]})$
satisfies the following properties: for any $X,Y\in\mathbf{L}_{G}^{p}(\mathbf{W})$,
$\eta\in\mathbf{L}_{G}^{p}(\mathbf{W}_{[0,t]})$, \end{prop}
\begin{description}
\item [{(i)}] $\hat{\mathbb{E}}[X\left|\mathcal{F}_{t}\right]\geq\hat{\mathbb{E}}[Y\left|\mathcal{F}_{t}\right]$\ for\ $X\geq Y$. 
\item [{(ii)}] $\hat{\mathbb{E}}[\eta\left|\mathcal{F}_{t}\right]=\eta$. 
\item [{(iii)}] $\mathbb{\hat{E}}[X+Y\left|\mathcal{F}_{t}\right]\leq\mathbb{\hat{E}}[X\left|\mathcal{F}_{t}\right]+\mathbb{\hat{E}}[Y\left|\mathcal{F}_{t}\right]$. 
\item [{(iv)}] $\mathbb{\hat{E}}[\eta X\left|\mathcal{F}_{t}\right]=\eta^{+}\mathbb{\hat{E}}[X\left|\mathcal{F}_{t}\right]+\eta^{-}\mathbb{\hat{E}}[-X\left|\mathcal{F}_{t}\right]$. 
\item [{(v)}] $\hat{\mathbb{E}}[\hat{\mathbb{E}}[X\left|\mathcal{F}_{t}\right]|\mathcal{F}_{s}]=\hat{\mathbb{E}}[X|\mathcal{F}_{t\wedge s}]$. \end{description}
\begin{prop}
For each $X,Y\in\mathbf{L}_{G}^{p}(\mathbf{W}_{[0,T]})$, $t\leq T$
, $\mathbb{\hat{E}}[Y\left|\mathcal{F}_{t}\right]=-\mathbb{\hat{E}}[-Y\left|\mathcal{F}_{t}\right]$.
Then 
\[
\mathbb{\hat{E}}[X+\alpha Y\left|\mathcal{F}_{t}\right]=\mathbb{\hat{E}}[X\left|\mathcal{F}_{t}\right]+\alpha\mathbb{\hat{E}}[Y\left|\mathcal{F}_{t}\right],\ \alpha\in\mathbb{R}.
\]

\end{prop}

\subsection{Stochastic integrals with respect to spatial-temporal white noise}

In the following, we define the stochastic integral with respect to
the spatial-temporal G-white noise $\mathbf{W}([s,t)\times A),0\leq s<t\leq T,A\in\mathcal{B}_{0}(\mathbb{R}^{d}).$

Firstly, let $\mathbf{M}^{0}([0,T]\times\mathbb{R}^{d})$ be the collection
of simple processes with the form: 
\[
f(s,x;\omega)=\sum\limits _{i=0}^{n-1}\sum_{j=1}^{m}X_{ij}(\omega)\textbf{1}_{A_{j}}(x)\textbf{1}_{[t_{i},t_{i+1})}(s),
\]
where $X_{ij}\in L_{ip}(\mathcal{F}_{t_{i}})$, $i=0,\cdots,n-1$,
$j=1,\cdots,m$, $0=t_{0}<t_{1}<\cdots<t_{n}=T$, and $\{A_{j}\}_{j=1}^{m}$
is a space partition of $\mathcal{B}_{0}(\mathbb{R}^{d})$, namely,
\[
A_{j}\in\mathcal{B}_{0}(\mathbb{R}^{d}),\,\,\,\bigcup_{j=1}^{m}A_{j}=\mathbb{R}^{d},\,\,\,and\,\,\,A_{j}\cap A_{k}=\emptyset,\,\,\text{ if }j\neq k,\ j,k=1,\cdots,m.
\]
The Bohner's integral of $f\in\mathbf{M}^{0}([0,T]\times\mathbb{R}^{d})$
is defined by 
\begin{align*}
I_{B}(f)=\int_{\mathbb{R}^{d}}\int_{0}^{T}f(s,x)dsdx= & \int_{0}^{T}\int_{\mathbb{R}^{d}}f(s,x)dxds:=\sum\limits _{i=0}^{n-1}\sum_{j=1}^{m}X_{ij}(t_{i+1}-t_{i})\lambda_{A_{j}},\,\,\,\forall t\in[0,T].%\quadB\in\mathcal{B}(\mathbb{R}^{d}).
\end{align*}
It is clear that $I_{B}:\mathbf{M}^{0}([0,T]\times\mathbb{R}^{d})\mapsto L_{ip}(\mathcal{F}_{T})$
is a linear mapping. For a given $p\geq1$ and $f\in\mathbf{M}^{0}([0,T]\times\mathbb{R}^{d})$,
let 
\[
\|f\|_{\mathbf{M}^{p}}=\left(\hat{\mathbb{E}}[\int_{0}^{T}\int_{\mathbb{R}^{d}}|f(s,x)|^{p}dsdx]\right)^{\frac{1}{p}}=\left\{ \hat{\mathbb{E}}\left[\sum_{i=0}^{n-1}\sum_{j=1}^{m}X_{ij}^{p}(t_{i+1}-t_{i})\lambda_{A_{j}}\right]\right\} ^{1/p}.
\]
Then $\|f\|_{\mathbf{M}^{p}}$ is a Banach norm on $\mathbf{M}^{0}([0,T]\times\mathbb{R}^{d})$
and the completion of $\mathbf{M}^{0}([0,T]\times\mathbb{R}^{d})$
under the norm $\|\cdot\|_{\mathbf{M}^{p}}$, denoted by $\mathbf{M}_{G}^{p}([0,T]\times\mathbb{R}^{d})$,
is a Banach space.

The stochastic integral with respect to the spatial-temporal indexed
white noise $\mathbf{W}$ from $\mathbf{M}^{0}([0,T]\times\mathbb{R}^{d})$
to $L_{ip}(\mathcal{F}_{T})$ can be defined as follows: 
\begin{equation}\label{Def-int}
\int_{0}^{T}\int_{\mathbb{R}^{d}}f(s,x)\mathbf{W}(ds,dx):=\sum\limits _{i=0}^{n-1}\sum_{j=1}^{m}X_{ij}\mathbf{W}([t_{i},t_{i+1})\times A_{j}).
\end{equation}

\begin{lem}
\label{pmW} For any $0\leq s\leq t<\infty$ and $\xi\in\mathbf{L}_{G}^{p}(\mathbf{W}_{[0,s]})$,
we have 
\[
\hat{\mathbb{E}}[\xi\mathbf{W}([s,t)\times A)]=0,\quad\forall A\in\mathcal{B}_{0}(\mathbb{R}),
\]
and 
\[
\hat{\mathbb{E}}[\xi\mathbf{W}([s,t)\times A)(\mathbf{W}([s,t)\times\bar{A})]=0,\qquad\forall A,\bar{A}\in\mathcal{B}_{0}(\mathbb{R})\mbox{ such that}\mbox{ }A\cap\bar{A}=\emptyset.
\]
\end{lem}
\begin{proof}
The first relation is obvious, since $\mathbf{W}([s,t)\times A)$
is independent of $\xi$ and $\hat{\mathbb{E}}[\pm\mathbf{W}([s,t)\times A)]=0$.
The second relation is derived from 
\[
\hat{\mathbb{E}}[\pm\mathbf{W}([s,t)\times A)(\mathbf{W}([s,t)\times\bar{A})]=\hat{\mathbb{E}}[\pm\mathbf{W}([s,t)\times(A\cap\bar{A}))]=0.
\]
\end{proof}
Using this lemma, we now provide the following lemma which allows
us to extend the the domain of stochastic integral to $\mathbf{M}_{G}^{2}([0,T]\times\mathbb{R}^{d})$:
\begin{lem}
\label{Wcontrol} For each $f\in\mathbf{M}^{0}([0,T]\times\mathbb{R}^{d})$,
\begin{align}
 & \hat{\mathbb{E}}\left[\int_{0}^{T}\int_{\mathbb{R}^{d}}f(s,x)\mathbf{W}(ds,dx)\right]=0.\label{W integral expectation 0 }\\
 & \hat{\mathbb{E}}\left[\left|\int_{0}^{T}\int_{\mathbb{R}^{d}}f(s,x)\mathbf{W}(ds,dx)\right|^{2}\right]\leq\overline{\sigma}^{2}\hat{\mathbb{E}}\left[\int_{0}^{T}\int_{\mathbb{R}^{d}}|f(s,x)|^{2}dsdx\right].\label{W continuity}
\end{align}
\end{lem}
\begin{proof}
Suppose $f(s,x;\omega)=\sum\limits _{i=0}^{n-1}\sum\limits _{j=1}^{m}X_{ij}(\omega)\textbf{1}_{A_{j}}(x)\textbf{1}_{[t_{i},t_{i+1})}(s),$
where $0\leq t_{0}<t_{1}<\cdots<t_{n}\leq T$, $\{A_{j}\}_{j=1}^{m}\subset\mathcal{B}_{0}(\mathbb{R}^{d})$
is a partition of $\mathbb{R}^{d}$, $X_{ij}\in L_{ip}(\mathcal{F}_{t_{i}})$,
$i=1,\cdots,n$, $j=1,\cdots,m$. It follows from the first relation
of Lemma \ref{pmW}, combined with Lemma \ref{pmLemma} that 
\[
\hat{\mathbb{E}}\left[\int_{0}^{T}\int_{\mathbb{R}^{d}}f(s,x)\mathbf{W}(ds,dx)\right]=\hat{\mathbb{E}}\left[\sum_{i=0}^{n-1}\sum_{j=1}^{m}X_{ij}(\mathbf{W}([t_{i},t_{i+1})\times A_{j})\right]=0.
\]
Hence we derive the equation (\ref{W integral expectation 0 }). For
the second one, 
\begin{align*}
\hat{\mathbb{E}}&\left[\left|\int_{0}^{T}\int_{\mathbb{R}^{d}}f(s,x)\mathbf{W}(ds,dx)\right|^{2}\right]  =\hat{\mathbb{E}}\left[\left|\sum_{i=0}^{n-1}\sum_{j=1}^{m}X_{ij}\mathbf{W}([t_{i},t_{i+1})\times A_{j})\right|^{2}\right]\\
 & =\hat{\mathbb{E}}\left[\sum_{i=0}^{n-1}\sum_{j=1}^{m}\sum_{l=0}^{n-1}\sum_{k=1}^{m}X_{ij}X_{lk}\mathbf{W}([t_{i},t_{i+1})\times A_{j})\mathbf{W}([t_{l},t_{l+1})\times A_{k})\right]\\
 & =\hat{\mathbb{E}}\left[\sum_{i=0}^{n-1}\sum_{j=1}^{m}\sum_{k=1}^{m}X_{ij}X_{ik}\mathbf{W}([t_{i},t_{i+1})\times A_{j})\mathbf{W}([t_{i},t_{i+1})\times A_{k})\right].
\end{align*}
where the last equality is from the first relation of Lemma \ref{pmW},
combined with Lemma \ref{pmLemma} since $A_{j}\cap A_{k}=\emptyset$
for $j\not=k$. We then apply the second relation of Lemma \ref{pmW},
once more combined with Lemma~\ref{pmLemma}, to obtain 
\begin{align*}
\hat{\mathbb{E}}\left[\left|\int_{0}^{T}\int_{\mathbb{R}^{d}}f(s,x)\mathbf{W}(ds,dx)\right|^{2}\right] & =\hat{\mathbb{E}}\left[\sum_{i=0}^{n-1}\sum_{j=1}^{m}X_{ij}^{2}\left(\mathbf{W}([t_{i},t_{i+1})\times A_{j})\right)^{2}\right]\\
 & \leq\hat{\mathbb{E}}\left[\eta_{n}+{\overline{\sigma}}^{2}\hat{\mathbb{E}}\left[\sum_{i=0}^{n-1}\sum_{j=1}^{m}X_{ij}^{2}(t_{i+1}-t_{i})\lambda_{A_{j}}\right]\right]\\
 & ={\overline{\sigma}}^{2}\hat{\mathbb{E}}\left[\sum_{i=0}^{n-1}\sum_{j=1}^{m}X_{ij}^{2}(t_{i+1}-t_{i})\lambda_{A_{j}}\right]\\
 & ={\overline{\sigma}}^{2}\hat{\mathbb{E}}\left[\int_{0}^{T}\int_{\mathbb{R}^{d}}|f(t,x)|^{2}dxdt\right],
\end{align*}
where we denote 
\[
\eta_{k}:=\sum_{i=0}^{k-1}\sum_{j=1}^{m}X_{ij}^{2}\left[\left(\mathbf{W}([t_{i},t_{i+1})\times A_{j})\right)^{2}-{\overline{\sigma}}^{2}(t_{i+1}-t_{i})\lambda_{A_{j}}\right],\,\,\,k=1,\cdots,n.
\]
Here, to obtain that $\hat{\mathbb{E}}[\eta_{n}]=0$, we take the
following procedure, adopted from Peng \cite{P2010}. 
\begin{align*}
\hat{\mathbb{E}}[\eta_{n}] & =\hat{\mathbb{E}}\left[\eta_{n-1}+\sum_{j=1}^{m}X_{n-1,j}^{2}\left[\left(\mathbf{W}([t_{n-1},t_{n})\times A_{j})\right)^{2}-{\overline{\sigma}}^{2}(t_{n}-t_{n-1})\lambda_{A_{j}}\right]\right]\\
 & =\hat{\mathbb{E}}\left[a+\hat{\mathbb{E}}\left[\sum_{j=1}^{m}b_{j}^{2}\Big[\left(\mathbf{W}([t_{n-1},t_{n})\times A_{j})\right)^{2}-{\overline{\sigma}}^{2}(t_{n}-t_{n-1})\lambda_{A_{j}}\Big]\right]_{a=\eta_{n-2},b_{j}=X_{n-1,j}}\right]\\
 & =\hat{\mathbb{E}}[\eta_{n-2}]=\cdots=\hat{\mathbb{E}}[\eta_{1}]=0.
\end{align*}
The proof is complete. 
\end{proof}
Therefore, we can continuously extend the domain $\mathbf{M}^{0}([0,T]\times\mathbb{R}^{d})$ of the stochastic integral  (\ref{Def-int}) to 
$\mathbf{M}^{2}_G([0,T]\times\mathbb{R}^{d})$. 
%with
%respect to the spatial-temporal G-white noise $\mathbf{W}$ to the
%mapping from $\mathbf{M}_{G}^{2}([0,T]\times\mathbb{R}^{d})$ to {\color{red}$\mathbf{L}_{G}^{2}(\mathbf{W}_{[0,T]})$}.
Indeed, for any $f\in\mathbf{M}_{G}^{2}([0,T]\times\mathbb{R}^{d})$, there
exists a sequence of simple processes $f_{n}\in\mathbf{M}^{0}([0,T]\times\mathbb{R}^{d})$
such that $\lim\limits _{n\rightarrow\infty}\|f_{n}-f\|_{\mathbf{M}^{2}}=0$.
By Lemma \ref{Wcontrol}, we have that 
\[
\left\{\int_{0}^{T}\int_{\mathbb{R}^{d}}f_{n}(s,x)\mathbf{W}(ds,dx)\right\}_{n=1}^{\infty}
\]
is a Cauchy sequence in $\mathbf{L}_{G}^{2}(\mathbf{W}_{[0,T]})$.
As $\mathbf{L}_{G}^{2}(\mathbf{W}_{[0,T]})$ is complete, we can define
\[
\int_{0}^{T}\int_{\mathbb{R}^{d}}f(s,x)\mathbf{W}(ds,dx):=\mathbb{L}^{2}-\lim\limits _{n\rightarrow\infty}\int_{0}^{T}\int_{\mathbb{R}^{d}}f_{n}(s,x)\mathbf{W}(ds,dx).
\]

It is easy to check that the stochastic integral has the following
properties. 
\begin{prop}
\label{stochastic integral properties} For each $f,g\in\mathbf{M}_{G}^{2}([0,T]\times\mathbb{R}^{d})$,
$0\leq s\leq r\leq t\leq T$, we have \end{prop}
\begin{itemize}
\item [{(1)}] $\int_{s}^{T}\int_{\mathbb{R}^{d}}f(s,x)\mathbf{W}(ds,dx)=\int_{s}^{r}\int_{\mathbb{R}^{d}}f(s,x)\mathbf{W}(ds,dx)+\int_{r}^{T}\int_{\mathbb{R}^{d}}f(s,x)\mathbf{W}(ds,dx).$ 
\item [{(2)}] $\int_{s}^{T}\int_{\mathbb{R}^{d}}(\alpha f(s,x)+g(s,x))\mathbf{W}(ds,dx)$\\ $ =\alpha\int_{s}^{T}\int_{\mathbb{R}^{d}}f(s,x)\mathbf{W}(ds,dx)+\int_{s}^{T}\int_{\mathbb{R}^{d}}g(s,x)\mathbf{W}(ds,dx)$,
if  $\alpha\in\mathbf{L}_{G}^{1}(\mathbf{W}_{[0,s]})$ is bounded.
\item [{(3)}] $\hat{\mathbb{E}}[\int_{r}^{T}\int_{\mathbb{R}^{d}}f(s,x)\mathbf{W}(ds,dx)|\mathcal{F}_{r}]=0.$ 
\end{itemize}
For each $f\in\mathbf{M}_{G}^{2}([0,T]\times\mathbb{R}^{d})$, by
(3) in Proposition \ref{stochastic integral properties}, we can obtain
that 
\[\{\int_{0}^{t}\int_{\mathbb{R}^{d}}f(r,x)\mathbf{W}(dr,dx),\,\,\,\, t\in[0,T]\}
\]
is a martingale in the sense that 
\[
\hat{\mathbb{E}}[\int_{0}^{t}\int_{\mathbb{R}^{d}}f(r,x)\mathbf{W}(dr,dx)|\mathcal{F}_{s}]=\int_{0}^{s}\int_{\mathbb{R}^{d}}f(r,x)\mathbf{W}(dr,dx)\ \ if\ 0\leq t\leq T.
\]
Furthermore, $\{\int_{0}^{T}\int_{\mathbb{R}^{d}}f(r,x)\mathbf{W}(dr,dx),\mathcal{F}_{t},t\in[0,T]\}$
is a martingale measure.

\section{Conclusion remarks}
Based on a new type of $G$-Gaussian process introduced in Peng \cite{Peng2011},
we develop a general Gaussian random field under a sublinear expectation
framework. We first construct a $G$-Gaussian random field whose finite
dimensional distributions are $G$-normally distributed. A remarkable
point is that, different from the classical probability theory, a
$G$-Brownian motion is not a $G$-Gaussian random field. A $G$-Brownian motion is suitable 
to be as a temporal random field, whereas a random
Gaussian field is rather suitable for describe spatial uncertainty.
This new framework of $G$-Gaussian random field can be used to quantitatively
and efficiently study such type of spatial uncertainties and risks.

We then define a special family of
generating functions $\{G_{\underline{\gamma}}\}_{\underline{\gamma}\in\mathcal{J}_{\Gamma}}$
as (\ref{g-q}) and  establish the corresponding white noise of
space type. Stochastic integral of functions in $L^{2}(\mathbb{R}^{d})$
with respect to the spatial white noise has been established and the
family of stochastic integrals is still a $G$-Gaussian random field.

Furthermore, we focus on the spatial-temporal $G$-white noise on the
sublinear expectation space. Applying the similar method of $G$-It\^o's
integral, we develop the stochastic calculus with respect to the spatial-temporal
white noise on $\mathbf{M}_{G}^{2}([0,T]\times\mathbb{R}^{d})$.

%{\color{red} 
%It is important to mention that a spatial-temporal $G$-white noise is no more a $G$-Gaussian random field. 
%But for a fixed $A\in \mathcal{B}_0(\mathbb{B})$, $W([0,t)\times A$ is a $G$-Brownian motion, and  for a fixed $0\leq s<t<\infty$, $\{W([s,t]\times A\}_{A\in \mathcal{B}_0(\mathbb{B})}$ is a spatial  Gaussian field. 
%
%On the other hand, for fixed time index, the spatial-temporal $G$-white noise is still a $G$-Gaussian random field and for fixed space index, the corresponding white noise becomes a $G$-Brownian motion.} 

$G$-Gaussian random field and $G$-white noise theory can be widely
used to study the uncertain models in many physical fields, such
as statistical physics and quantum theory.

\section{Appendix}

\subsection{Upper expectation and the corresponding capacity}

Let $\Gamma$ be defined as (\ref{Gamma}). In this section, unless
otherwise mentioned, we always denote by $\Omega=C_{0}(\Gamma,\mathbb{R})$,
the space of all $\mathbb{R}-$valued continuous functions $(\omega_{\gamma})_{\gamma\in\Gamma}$,
with $\omega_{0}=0$, equipped with the distance 
\[
\rho(\omega^{(1)},\omega^{(2)}):=\sum_{i=1}^{\infty}2^{-i}[(\max_{\gamma\subset\mathfrak{B}^{d}(0;i)}|\omega_{\gamma}^{(1)}-\omega_{\gamma}^{(2)}|)\wedge1],\ \forall\omega^{(1)},\omega^{(2)}\in\Omega,
\]
where $\mathfrak{B}^{d}(0;i)=\{x\in\mathbb{R}^{d}:|x|\leq i\}$. Let$\ \bar{\Omega}=\mathbb{R}^{\Gamma}$
denote the space of all $\mathbb{R}-$valued functions $(\bar{\omega}_{\gamma})_{\gamma\in\Gamma}$.

We also denote by $\mathcal{B}(\Omega)$, the $\sigma$-algebra generated
by all open sets and let $\mathcal{B}(\bar{\Omega})$ be the $\sigma$-algebra
generated by all finite dimensional cylinder sets. The corresponding
canonical process is $\mathbb{W}_{\gamma}(\omega)=\omega_{\gamma}$
(respectively, $\bar{\mathbb{W}}_{\gamma}(\bar{\omega})=\bar{\omega}_{\gamma}$),
$\gamma\in\Gamma$ for $\omega\in\Omega$\ (respectively,\ $\bar{\omega}\in\bar{\Omega}$).

The spaces of Lipschitz cylinder functions on $\Omega$ and $\bar{\Omega}$
are denoted respectively by 
\[
L_{ip}(\Omega):=\{\varphi({\mathbb{W}}_{{\gamma}_{1}},{\mathbb{W}}_{{\gamma}_{2}},\cdots,{\mathbb{W}}_{{\gamma}_{n}}):\forall n\geq1,\gamma_{1},\cdots,\gamma_{n}\in\Gamma,\forall\varphi\in C_{l.Lip}(\mathbb{R}^{n})\},
\]
\[
L_{ip}(\bar{\Omega}):=\{\varphi(\bar{{\mathbb{W}}}_{{\gamma}_{1}},\bar{{\mathbb{W}}}_{{\gamma}_{2}},\cdots,\bar{{\mathbb{W}}}_{{\gamma}_{n}}):\forall n\geq1,\gamma_{1},\cdots,\gamma_{n}\in\Gamma,\forall\varphi\in C_{l.Lip}(\mathbb{R}^{n})\}.
\]

Let $\{G_{\underline{\gamma}}(\cdot),\underline{\gamma}\in\mathcal{J}_{\Gamma}\}$
be a family of continuous monotone and sublinear functions given by
(\ref{g-q}). Following Section 4, we can construct the corresponding
sublinear expectation $\mathbb{\bar{E}}$ on $(\bar{\Omega},L_{ip}(\bar{\Omega}))$.
Due to the natural correspondence of $L_{ip}(\bar{\Omega})$ and $L_{ip}(\Omega)$,
we also construct a sublinear expectation $\hat{\mathbb{E}}$ on $(\Omega,L_{ip}(\Omega))$
such that $({\mathbb{W}}_{{\gamma}}({\omega}))_{\gamma\in\Gamma}$
is a $G$-white noise.

We want to construct a weakly compact family of ($\sigma$-additive)
probability measures $\mathcal{P}$ on $(\Omega,{\mathcal{B}}(\Omega))$
such that the sublinear expectation can be represented as an upper
expectation, namely, 
\[
\hat{\mathbb{E}}[\cdot]=\max_{P\in\mathcal{P}}E_{P}[\cdot].
\]
The following lemmas are variations of Lemma 3.3 and 3.4 in Chapter
I of Peng \cite{P2010}.

We denote $\mathcal{J}_{\Gamma}:=\{\underline{{\gamma}}=(\gamma_{1},\ldots,\gamma_{m}):\forall m\in\mathbb{N},\,\,\gamma_{i}\in\Gamma\}.$ 
\begin{lem}
\label{le3} Let $\underline{\gamma}=(\gamma_{1},\cdots,\gamma_{m})\in\mathcal{J}_{\Gamma}$
and $\{\varphi_{n}\}_{n=1}^{\infty}\subset C_{l.Lip}(\mathbb{R}^{m})$
satisfy $\varphi_{n}\downarrow0$ as $n\to\infty$. Then $\mathbb{\bar{E}}[\varphi_{n}(\bar{{\mathbb{W}}}_{{\gamma}_{1}},\bar{{\mathbb{W}}}_{{\gamma}_{2}},\cdots,\bar{{\mathbb{W}}}_{{\gamma}_{m}})]\downarrow0$.
\end{lem}
\begin{lem}
\label{le4} Let $E$ be a finitely additive linear expectation dominated by $\mathbb{\bar{E}}$ on $L_{ip}(\bar{\Omega})$.
Then there exists a unique probability measure $P$ on $(\bar{\Omega},\mathcal{B}(\bar{\Omega}))$
such that $E[X]=E_{P}[X]$ for each $X\in L_{ip}(\bar{\Omega})$. \end{lem}
\begin{proof}
For any fixed $\underline{\gamma}=(\gamma_{1},\ldots,\gamma_{m})\in\mathcal{J}_{\Gamma}$,
by Lemma \ref{le3}, for each sequence $\{\varphi_{n}\}_{n=1}^{\infty}\subset C_{l.Lip}(\mathbb{R}^{m})$
satisfying $\varphi_{n}\downarrow0$, we have $\mathbb{\bar{E}}[\varphi_{n}(\bar{\mathbb{W}}_{{\gamma}_{1}},\bar{\mathbb{W}}_{{\gamma}_{2}},\cdots,\bar{\mathbb{W}}_{{\gamma}_{m}})]\downarrow0$.
By Daniell-Stone's theorem (see Appendix B in Peng \cite{P2010}),
there exists a unique probability measure $P_{\underline{{\gamma}}}$
on $(\mathbb{R}^{m},\mathcal{B}(\mathbb{R}^{m}))$ such that $E_{P_{\underline{{\gamma}}}}[\varphi]=\mathbb{\bar{E}}[\varphi(\bar{\mathbb{W}}_{{\gamma}_{1}},\bar{\mathbb{W}}_{{\gamma}_{2}},\cdots,\bar{\mathbb{W}}_{{\gamma}_{m}})]$
for each $\varphi\in C_{l.Lip}(\mathbb{R}^{m})$. Thus we get a family
of finite dimensional distributions $\{P_{\underline{{\gamma}}}:\underline{{\gamma}}\in\mathcal{J}_{\Gamma}\}$.
It is easy to check that $\{P_{\underline{\gamma}}:\underline{{\gamma}}\in\mathcal{J}_{\Gamma}\}$
is a consistent family. Then by Kolmogorov's consistent extension
theorem, there exists a probability measure $P$ on $(\bar{\Omega},\mathcal{B}(\bar{\Omega}))$
such that $\{P_{\underline{{\gamma}}}:\underline{{\gamma}}\in\mathcal{J}_{\Gamma}\}$
is the finite dimensional distributions of $P$. We now prove the
uniqueness. Assume that there exists another probability measure $\bar{P}$
satisfying the condition. By Daniell-Stone's theorem, $P$ and $\bar{P}$
have the same finite-dimensional distributions, hence by the monotone
class theorem, $P=\bar{P}$. The proof is complete. \end{proof}
\begin{lem}
\label{le5} There exists a family of probability measures $\mathcal{P}_{e}$
on $(\bar{\Omega},\mathcal{B}(\bar{\Omega}))$ such that 
\[
\mathbb{\bar{E}}[X]=\max_{P\in\mathcal{P}_{e}}E_{P}[X],\quad\text{for}\ X\in L_{ip}(\bar{\Omega}).
\]
\end{lem}
\begin{proof}
By the representation theorem of the sublinear expectation and Lemma~\ref{le4},
it is easy to get the result. 
\end{proof}
\ \ \
 For this $\mathcal{P}_{e}$, we define the associated capacity: 
\[
\tilde{c}(A):=\sup_{P\in\mathcal{P}_{\!\!e}}P(A),\quad A\in\mathcal{B}(\bar{\Omega}),
\]
and the upper expectation for each $\mathcal{B}(\bar{\Omega})$-measurable
real function $X$ which makes the following definition meaningful:
\[
\mathbb{\tilde{E}}[X]:=\sup_{P\in\mathcal{P}_{\!\!e}}E_{P}[X].
\]

\ \ \
By Theorem \ref{modification1}, there exists a continuous modification
$\tilde{\mathbb{W}}$ of $\bar{\mathbb{W}}$ in the sense that $\tilde{c}(\{\tilde{\mathbb{W}}_{{\gamma}}\not=\bar{\mathbb{W}}_{{\gamma}}\})=0$,
and 
\[
\mathbb{\tilde{E}}[|\tilde{\mathbb{W}}_{\gamma_{1}}-\tilde{\mathbb{W}}_{\gamma_{2}}|^{2d+2}]=\mathbb{\tilde{E}}[|\bar{\mathbb{W}}_{\gamma_{1}}-\bar{\mathbb{W}}_{\gamma_{2}}|^{2d+2}]=\mathbb{\bar{E}}[|\bar{\mathbb{W}}_{{\gamma}_{1}}-\bar{\mathbb{W}}_{{\gamma}_{2}}|^{2d+2}]=L|\gamma_{1}-\gamma_{2}|^{d+1},\forall\gamma_{1},\gamma_{2}\in\Gamma,
\]
where $L$ is a constant depending only on $\bar{E}$.

For any $P\in\mathcal{P}_{\!\!e}$, let $P\circ\tilde{\mathbb{W}}^{-1}$
denote the probability measure on $(\Omega,\mathcal{B}(\Omega))$
induced by $\tilde{\mathbb{W}}$ with respect to $P$. We denote $\mathcal{P}_{1}=\{P\circ\tilde{\mathbb{W}}^{-1}:P\in\mathcal{P}_{\!\!e}\}$.
Applying the well-known criterion for tightness of Kolmogorov-Chentsov's
type expressed in terms of moments (see Appendix B in Peng \cite{P2010}),
we conclude that $\mathcal{P}_{1}$ is tight. We denote by $\mathcal{P}=\overline{\mathcal{P}}_{1}$
the closure of $\mathcal{P}_{1}$ under the topology of weak convergence,
then $\mathcal{P}$ is weakly compact.

Now, we give the representation of the sublinear expectation. 
\begin{thm}
\label{Gt34} Let $\{G_{\underline{\gamma}}(\cdot),\underline{\gamma}\in\mathcal{J}_{\Gamma}\}$
be a family of continuous monotone and sublinear functions given by
(\ref{g-q}) and $\hat{\mathbb{E}}$ be the corresponding sublinear
expectation on $(\Omega,L_{ip}(\Omega))$. Then there exists a weakly
compact family of probability measures $\mathcal{P}$ on $(\Omega,\mathcal{B}(\Omega))$
such that 
\end{thm}
\[
\hat{\mathbb{E}}[X]=\max_{P\in\mathcal{P}}E_{P}[X]\quad\text{for}\ X\in L_{ip}(\Omega).
\]

\begin{proof}
By Lemma \ref{le5}, we have %and Theorem \ref{modification1},  
\[
\hat{\mathbb{E}}[X]=\max_{P\in\mathcal{P}_{1}}E_{P}[X]\quad\text{for}\ X\in L_{ip}(\Omega).
\]
For any $X\in L_{ip}(\Omega)$, by Lemma \ref{le3}, we get $\mathbb{\hat{E}}[|X-(X\wedge N)\vee(-N)|]\downarrow0$
as $N\rightarrow\infty$. Noting that $\mathcal{P}=\overline{\mathcal{P}}_{1}$,
by the definition of weak convergence, we thus have the result.
\end{proof}

\end{document}